\documentclass[a4paper,oneside,10pt]{article}%
\usepackage{amsmath}
\usepackage{amsfonts}
\usepackage{amssymb}
\usepackage{multirow}
\usepackage{graphicx}
\usepackage{epstopdf}
\usepackage{epsfig}
\usepackage{subfigure}
\usepackage[square,numbers,sort&compress]{natbib}%
\providecommand{\U}[1]{\protect\rule{.1in}{.1in}}
\providecommand{\U}[1]{\protect\rule{.1in}{.1in}}
\allowdisplaybreaks[4]
\allowdisplaybreaks[4]

\newtheorem{theorem}{Theorem}[section]

\newtheorem{definition}[theorem]{Definition}

\newtheorem{lemma}[theorem]{Lemma}

\newtheorem{proposition}[theorem]{Proposition}
\newtheorem{remark}[theorem]{Remark}

\newenvironment{proof}[1][Proof]{\noindent\textbf{#1.} }{\ \rule{0.5em}{0.5em}}
\numberwithin{equation}{section}

\begin{document}

\title{Novel multi-step predictor-corrector schemes for backward stochastic
differential equations}
\author{Qiang Han \thanks{Zhongtai Securities Institute for Financial Studies,
Shandong University, Jinan, Shandong 250100, PR China.
201411272@mail.sdu.edu.cn. }
\and Shaolin Ji\thanks{Zhongtai Securities Institute for Financial Studies,
Shandong University, Jinan, Shandong 250100, PR China. jsl@sdu.edu.cn
(Corresponding author). Research supported by National Natural Science of
China No. 11971263; 11871458 and National Key R\&D Program of China
No.2018YFA0703900.}}
\maketitle


\textbf{Abstract}. Novel multi-step predictor-corrector numerical schemes have
been derived for approximating decoupled forward-backward stochastic
differential equations (FBSDEs).
The stability and high order rate of convergence of the schemes are rigorously
proved. We also present a sufficient and necessary condition for the stability
of the schemes. Numerical experiments are given to illustrate the stability
and convergence rates of the proposed methods.



\addcontentsline{toc}{section}{\hspace*{1.8em}Abstract}

\section{Introduction}

To the best of our knowledge, the numerical algorithms for decoupled FBSDEs
can be divided into two:

One branch explores the connection with partial differential equations (PDEs).
To be specific, the solution $(Y_{t},Z_{t})$ of the BSDE in (\ref{2-1}) can be
represented as $Y_{t} = u(t,x_{t}),$ $Z_{t} = \nabla_{x}u(t,x_{t}%
)\sigma(t,x_{t}),$ $t\in[0,T],$ $u\in C_{b}^{1,2}([0,T]\times\mathbb{R}^{d}),$
and $u(t,x_{t})$ is solution of the parabolic PDE
\[
\frac{\partial u(t,x)}{\partial t}+\sum\limits_{i=1}^{d}b_{i}\frac{\partial
u(t,x)}{\partial x_{i}}+ \frac{1}{2}\sum\limits_{i,j=1}^{d}(\sigma\sigma
^{\top})_{ij}\frac{\partial^{2}u(t,x)}{\partial x_{i}\partial x_{j}%
}u(t,x)+f(t,x,u(t,x),\nabla_{x}u(t,x)\sigma(t,x))=0,
\]
with the terminal condition $u(T,x) = \Phi(x).$ In turn, suppose $(Y_{t}%
,Z_{t})$ is the solution of the BSDE in (\ref{2-1}). $u(t,x_{t}) = Y_{t}$ is a
viscosity solution to the PDE. Thus, the numerical approximation of decoupled
FBSDEs is to solve the corresponding parabolic PDEs numerically (see
\cite{13,29,30}). This algorithm may be limited due to high-dimensionality or
lack of smoothness of the coefficients. For this issue, Weinan E et al propose
the deep learning algorithm which can deal with 100-dimensional nonlinear PDEs
(see \cite{3,14,WEMHAJTK19,24,MHAJTK19}). Also, the branching diffusion method
does not suffer from the curse of dimensionality (see \cite{25}) and this
method is extended to the non-Markovian case and the non-linearities case (see
\cite{27} and \cite{26} respectively).

The second branch of algorithms can be implemented via a two-step procedure
which consists of a time-discretization of decoupled FBSDEs and an
approximation procedure for the conditional expectations. Specifically, if the
Euler scheme (explicit, implicit or generalized) is utilized to discretize
decoupled FBSDEs, the order of discretization error is $\frac{1}{2}$ and
sometimes can reach $1$ (see \cite{1,4,7,12,18,20,21,28,37,JZ04}). To obtain
high order accuracy scheme, authors in \cite{38,39} develop two kinds of
multi-step schemes to solve decoupled FBSDEs. The Runge-Kutta schemes and
linear multi-step schemes for approximating decoupled FBSDEs have been
investigated in \cite{10} and \cite{9}. In this paper we extend the
predictor-corrector type based on Adams schemes (see \cite{JFC13}) to the
predictor-corrector type based on general linear multi-step schemes. We also
provide an indicator for the local truncation error by utilizing the
difference between the predicted and the corrected values at each time step
(see Proposition 3.2). Furthermore, we present a sufficient and necessary
condition for the stability of the general scheme (see Theorem 3.6).
Finally, parameters in multi-step schemes are obtained by different methods.
That is to say, the paper \cite{38} adopts derivative approximation; papers
\cite{9,JFC13,39} use Lagrange interpolating polynomials; and we utilize
It\^{o}-Taylor expansion.

From the above review, the time-discretization of decoupled FBSDEs can adopt
low order schemes or high order schemes. Notice that there are a large number
of documents about low order schemes and this implies that the theory of
implementable numerical methods of decoupled FBSDEs is booming. Compared with
the development of the numerical methods of ordinary differential equations
(ODEs) and stochastic differential equations (SDEs), the investigation of high
order accuracy schemes for decoupled FBSDEs is meaningful and necessary.
Moreover, the analysis of Section 4.4 of \cite{20} also maintains that
development of high order accuracy schemes for decoupled FBSDEs is
significant. Hence, for this motivation, we design an available high order
accuracy scheme called the general multi-step predictor-corrector schemes (see
(\ref{2-32})) (see \cite{8} about SDEs which do not have the predictor term).
And this kind of schemes possess the advantage of simple type of error
estimates for decoupled FBSDEs.

The contributions of this paper are as follows.

First, we derive a novel high order scheme for decoupled FBSDEs. The advantage
does not require the solution of an algebraic equation at each step.
Therefore, this can reduce the complexity of calculation. Simultaneously, our
schemes also inherit the virtues of implicit scheme. Second, the stability and
high order property of the scheme (\ref{2-32}) are rigorously proved. Note
that we present a sufficient and necessary condition for the stability of the
scheme (\ref{2-32}). A property of predictor-corrector scheme (see Proposition
3.2) is established in the frame of decoupled FBSDEs. And this property
provides an indicator for the local truncation error by utilizing the
difference between the predicted and the corrected values at each time step.
The high order property of the scheme (\ref{2-32}) is also established.

The structure of this paper is as follows. In Section 2, we present some
fundamental definitions, assumptions and lemmas that can be used in the
following sections. Moreover the Adams schemes of decoupled FBSDEs are
reviewed. We first construct the predictor-corrector schemes (\ref{2-32}).
Then, the stability and high order properties of scheme (\ref{2-32}) are also
found in Section 3. Section 4 presents numerical experiments to illustrate the
stability and convergence rates of algorithms.

\section{Preliminaries}

In this section, we provide some preliminary results and recall the
predictor-corrector scheme of decoupled FBSDEs based on Adams types.

\subsection{Decoupled FBSDE}

In this subsection, we review the decoupled FBSDE and the corresponding propositions.

Let $T>0$ be a fixed terminal time and $(\Omega,\mathcal{F},\mathbb{F}%
,\mathbb{P})$ be a filtered complete probability space where $\mathbb{F}%
=(\mathcal{F}_{t})_{0\le t\le T}$ is the natural filtration of the standard
$d$-dimensional Brownian motion. In the space $(\Omega,\mathcal{F}_{T},
\mathbb{F},\mathbb{P})$, we consider discretizing the decoupled FBSDEs as
below:
\begin{equation}
\left\{
\begin{array}
[c]{rl}%
X_{t}= & x_{0} + \int_{0}^{t} b(s,X_{s})ds + \int_{0}^{t}\sigma(s,X_{s})d
W_{s},~~~~~~~~~~~~~~~~~(SDE)\\
Y_{t}= & \Phi(X_{T}) + \int_{t}^{T} f(s,X_{s},Y_{s},Z_{s})ds - \int_{t}^{T}
Z_{s} d W_{s},~~~~~~(BSDE)
\end{array}
\right.  \label{2-1}%
\end{equation}
where $(X_{s})_{t\le s\le T}$ is a $d$-dimensional diffusion process driven by
the finite $d$-dimensional Brownian motion $(W_{t})_{0\leq t\leq T}$ which is
defined in a filtered complete probability space $(\Omega,\mathcal{F}%
,\mathbb{F},\mathbb{P})$. Set the $\sigma$-algebra $\mathcal{F}_{t,s}%
=\sigma\{W_{r}-W_{t},t\leq r\leq s\},\mathcal{F}=\mathcal{F}_{0,T}.$ In
addition, functions $b,\sigma,\Phi$ and $f$ satisfy:\newline\textbf{Assumption
1.} There exists a non-negative constant $L$ satisfying
\[
|b(t,x_{1})-b(t,x_{2})|+|\sigma(t,x_{1})-\sigma(t,x_{2})|\le L|x_{1}%
-x_{2}|,~~~\forall x_{1},x_{2}\in\mathbb{R}^{d}.
\]
\textbf{Assumption 2.} There exist non-negative constants $C_{f}$ and $L_{f}$
such that

\begin{description}
\item[(i)] $|f(t_{1},x_{1},y_{1},z_{1})-f(t_{2},x_{2},y_{2},z_{2})|\le
L_{f}(\sqrt{|t_{1}-t_{2}|}+|x_{1}-x_{2}|+|y_{1}-y_{2}|+\|z_{1}-z_{2}\|)$ for
all $t_{1}, t_{2}\in[0,T], x_{1}, x_{2}\in\mathbb{R}^{d}$, $y_{1}, y_{2}%
\in\mathbb{R}$ and $z_{1},z_{2}\in\mathbb{R}^{d}$;

\item[(ii)] $|f(t,x,0,0)|\le C_{f}$ on $[0,T]\times\mathbb{R}^{d}$;

\item[(iii)] Function $\Phi$ is measurable and bounded.
\end{description}

For readers' convenience, here we present two lemmas and adapt them to our context.

\begin{lemma}
(see \cite{35}) Assume that functions $b, \sigma, f$ and $\Phi$ are uniformly
Lipschitz with respect to (w.r.t.) $(x,y,z)$ and $\frac{1}{2}$-H\"{o}lder
continuous w.r.t. t. In addition, assume $\Phi$ is of class $C_{b}^{2+\kappa}$
for some $\kappa\in(0,1)$ and the matrix valued function $a = \sigma
\sigma^{\top}= (a_{ij})$ is uniformly elliptic. Then the solution
$(Y_{t},Z_{t})$ of the BSDE in (\ref{2-1}) can be represented as
\[
Y_{t} = u(t,X_{t}),~~~~~~Z_{t} = \nabla_{x}u(t,X_{t})\sigma(t,X_{t}
),~~~~~~t\in[0,T],
\]
where $u\in C_{b}^{1,2}([0,T]\times\mathbb{R}^{d})$ satisfies the parabolic
PDE as below:
\begin{equation}
\mathcal{L}^{(0)}u(t,x)+f(t,x,u(t,x),\nabla_{x}u(t,x)\sigma(t,x))=0,
\label{2-2}%
\end{equation}
with the terminal condition $u(T,x) = \Phi(x)$ where $\mathcal{L}^{(0)}%
=\frac{\partial}{\partial t}+\sum_{i=1}^{d}b_{i}\frac{\partial}{\partial
x_{i}}+ \frac{1}{2}\sum_{i,j=1}^{d}(\sigma\sigma^{\top})_{ij}\frac
{\partial^{2}}{\partial x_{i}\partial x_{j}}$.
\end{lemma}

\begin{lemma}
(see Proposition 2.2 in \cite{10}) Let $n\ge0$. Then for a function
$v\in\mathcal{A}_{b}^{n+1},$
\begin{equation}
\mathbb{E}_{t}[v(t+h,X_{t+h})]=v_{t}+hv_{t}^{(0)}+\frac{h^{2}}{2}v_{t}%
^{(0,0)}+\cdots+\frac{h^{n}}{n!}v_{t}^{(0)_{n}}+O(h^{n+1}), \nonumber
\end{equation}

\end{lemma}

where $\mathbb{E}_{t}[\cdot]=\mathbb{E}[\cdot|\mathcal{F}_{t}]$;
$v_{t}^{\alpha}=v^{\alpha}(t,X_{t})$; $\mathcal{A}^{n}_{b},n\ge1$ is the set
of functions $v: [0,T]\times\mathbb{R}^{d}\rightarrow\mathbb{R}$ such that
$v\in\mathcal{A}^{\alpha}_{b}$ for all multi-index with finite length
$\alpha\in\{\alpha|\ell(\alpha)\le n\}\setminus\{\oslash\}$ is well defined,
continuous and bounded; $\mathcal{A}^{\alpha}_{b}$ denotes the subset of all
functions $v\in\mathcal{A}^{\alpha}$ such that the function $\mathcal{L}%
^{\alpha}v$ is bounded; $\mathcal{A}^{\alpha}$ is the set of all functions
$v:[0,T]\times\mathbb{R}^{d} \rightarrow\mathbb{R}$ for which $\mathcal{L}%
^{\alpha}v$ is well defined and continuous; $\ell(\alpha)$ is the length of a
multi-index of $\alpha$; let $v^{(0)}=\mathcal{L}^{(0)}v,v^{(0,0)}%
=\mathcal{L}^{(0)}\circ\mathcal{L}^{(0)}v,\cdots,$ $v^{(0)_{n}}=%
\begin{matrix}
\underbrace{\mathcal{L}^{(0)}\circ\cdots\circ\mathcal{L}^{(0)}}_{n}%
\end{matrix}
v.$

\subsection{Predictor-corrector discretization of the BSDE via Adams types}

In this subsection, for readers' convenience to understanding the following
text, we review the predictor-corrector discrete-time approximations of BSDE
with respect to $Y$ by Adams types (see \cite{JFC13}). As for the
time-discretization of $Z$, we adopt the scheme proposed in \cite{38}.

Before approximating solutions of the BSDEs, we first define a uniform
partition $\pi=\{t_{0}:=0<t_{1}<t_{2}\cdots<t_{N}:=T\}$ and the step
$h=\frac{T}{N}$, $\Delta W_{i}= W_{t_{i+1}}- W_{t_{i}},W_{i}= W_{t_{i}}$. We
consider the classical Euler discretization $X^{\pi}$ of the SDE
\begin{equation}
\left\{
\begin{array}
[c]{rl}%
X_{i+1}^{\pi}= & X_{i}^{\pi}+h b(t_{i},X_{i}^{\pi}) +\sigma(t_{i},X_{i}^{\pi
})\Delta W_{i},~~~i=0,1,\cdots,N-1,\\
X_{0}^{\pi}= & x_{0}.
\end{array}
\right. \nonumber
\end{equation}
It is known that $\sup\limits_{0\le i\le N}\mathbb{E}[|X_{t_{i}}-X_{i}^{\pi
}|^{2}]\rightarrow0$, as $h\rightarrow0$.

For non-stiff problems, Adams type is the most important linear multi-step
method. Its solution approximation at $t_{i}$ is defined either as
\begin{equation}
Y_{i}^{\pi}= \mathbb{E}_{i}\Big[Y_{i+1}^{\pi}+ h\sum_{\ell=1}^{k}\beta_{\ell
}f^{\pi}_{i+\ell}\Big], \label{2-10}%
\end{equation}
or as
\begin{equation}
Y_{i}^{\pi}= \mathbb{E}_{i}\Big[Y_{i+1}^{\pi}+ h\beta_{0}f_{i}^{\pi}+
h\sum_{\ell=1}^{k}\beta_{\ell}f^{\pi}_{i+\ell}\Big], \label{2-11}%
\end{equation}
where $Y_{i}^{\pi}$ and $Z_{i}^{\pi}$ denote the discretization form of $Y$
and $Z$ at $t_{i}$ and $f_{i}^{\pi}=f(t_{i},X_{i}^{\pi},Y_{i}^{\pi},Z_{i}%
^{\pi})$, $i=0,1,\cdots,N$; $\mathbb{E}_{i}[\cdot]=\mathbb{E}_{t_{i}}[\cdot]$;
$\beta_{0}\neq0$ and $\{\beta_{\ell}\}_{1\le\ell\le k}$ are real numbers and
$k\in\mathbb{N}^{+}$.

If we utilize the equation (\ref{2-11}) as the time-discretization of $Y$, we
are required the solution of an algebraic equation at each step because the
equation (\ref{2-11}) is implicit. To solve $Y$ in an explicit way, we can
first approximate $Y$ by the equation (\ref{2-10}). Now, the obtained value of
$Y$ is denoted as $\widetilde{Y}_{i}^{\pi}$, namely
\begin{equation}
\widetilde{Y}^{\pi}_{i} = \mathbb{E}_{i}[Y^{\pi}_{i+1} +h\sum\limits_{j=1}%
^{\widetilde{k}}\widetilde{\beta}_{j}f^{\pi}_{i+j}],\label{2-3}%
\end{equation}
where $\widetilde{k}\in\mathbb{N}^{+}$; $\widetilde{\beta}_{1}%
,\widetilde{\beta}_{2},\cdots,\widetilde{\beta}_{\widetilde{k}}$ are constants
and would be given in the following. Next, we use the improved equation
(\ref{2-11}) to approximate $Y$, namely
\begin{equation}
Y_{i}^{\pi}= \mathbb{E}_{i}[Y_{i+1}^{\pi}+ h \beta_{0}\widetilde{f}_{i}^{\pi}+
h\sum\limits_{j=1}^{k}\beta_{j}f^{\pi}_{i+j}], \label{2-4}%
\end{equation}
where $\widetilde{f}_{i}^{\pi}=f(t_{i},X_{i}^{\pi},\widetilde{Y}_{i}^{\pi
},Z_{i}^{\pi})$, for $i=N-1,N-2,\cdots,0$.

Next, we review the time-discretization of $Z$. From the BSDE in (\ref{2-1}),
we know
\begin{equation}
Y_{t_{i}}= Y_{r} + \int_{t_{i}}^{r} f(s,X_{s},Y_{s},Z_{s})ds -\int_{t_{i}}^{r}
Z_{s} d W_{s},~~~~~~r\in[t_{i},T]. \label{2-22}%
\end{equation}
Multiplying the above equation by $(W_{r}-W_{i})^{\top}$, and taking
conditional expectation, we obtain
\begin{equation}
0 = \mathbb{E}_{i}\Big[Y_{r}(W_{r}-W_{i})^{\top}\Big] + \int_{t_{i}}^{r}
\mathbb{E}_{i}\Big[f(s,X_{s},Y_{s},Z_{s})(W_{r}-W_{i})^{\top}\Big] ds -
\int_{t_{i}}^{r} \mathbb{E}_{i}[ Z_{s}] ds. \label{2-23}%
\end{equation}
Differentiating the equation (\ref{2-23}) w.r.t. $r$, we have
\begin{equation}
\frac{d\mathbb{E}_{i}\Big[Y_{r}(W_{r}-W_{i})^{\top}\Big]}{dr} = -\mathbb{E}%
_{i}\Big[f(r,X_{r},Y_{r},Z_{r})(W_{r}-W_{i})^{\top}\Big] + \mathbb{E}%
_{i}[Z_{r}]. \label{2-24}%
\end{equation}

Let $u\in C_{b}^{m+1}$, we apply Taylor's expansion at $t_{i}$ for function
$u(t)$, that is, for $n =0,1,2,\cdots,m$
\begin{equation}
u(t_{i} + nh) = u(t_{i}) + nhu^{\prime}(t_{i}) + \frac{(nh)^{2}}{2!}%
u^{\prime\prime}(t_{i}) +\frac{(nh)^{3}}{3!}u^{\prime\prime\prime}%
(t_{i})+\cdots. \label{2-25}%
\end{equation}
Moreover,
\begin{equation}
\sum_{n=0}^{m}\lambda_{m,n}u(t_{i} + nh) = \sum_{j=0}^{2}\frac{\sum
\limits_{n=0}^{m}\lambda_{m,n}(nh)^{j}}{j!} \frac{d^{j}u}{dt^{j}}(t_{i}) +
\mathcal{O}\Big(\sum_{n=0}^{m}\lambda_{m,n}(nh)^{m+1}\Big), \label{2-26}%
\end{equation}
where $\lambda_{m,0},\lambda_{m,1},\cdots,\lambda_{m,n}$ are real numbers. Let
$\lambda_{m,n},n =0,1,2,\cdots,m$ such that
\begin{equation}
\frac{1}{j!}\sum_{n=0}^{m}\lambda_{m,n}(nh)^{j}= \left\{
\begin{array}
[c]{rl}%
1, & j=1,\\
0, & j\neq1.
\end{array}
\right.  \label{2-27}%
\end{equation}
Hence, we deduce
\begin{equation}
\frac{du}{dt}(t_{i}) = \sum_{n=0}^{m}\lambda_{m,n}u(t_{i} + nh)+
\mathcal{O}\Big(\sum_{n=0}^{m}\lambda_{m,n}(nh)^{m+1}\Big). \label{2-28}%
\end{equation}
From the above equation, we have
\begin{equation}
\frac{d\mathbb{E}_{i}\big[Y_{r}(W_{r}-W_{i})^{\top}\big]}{dr}\Bigg|_{r =
t_{i}} =\sum_{n=0}^{m}\lambda_{m,n}\mathbb{E}_{i}\big[Y_{t_{i+n}}%
(W_{i+n}-W_{i})^{\top}\big]+R_{Z,i}, \label{2-29}%
\end{equation}
where $R_{Z,i} = \frac{d\mathbb{E}_{i}\big[Y_{r}(W_{r}-W_{i})^{\top}\big]}%
{dr}\Bigg|_{r = t_{i}} - \sum\limits_{n=0}^{m}\lambda_{m,n}\mathbb{E}%
_{i}\big[Y_{t_{i+n}}(W_{i+n}-W_{i})^{\top}\big]$. Combining (\ref{2-24}) with
(\ref{2-29}), we obtain
\begin{equation}
Z_{t_{i}} =\sum_{n=1}^{m}\lambda_{m,n}\mathbb{E}_{i}\big[Y_{t_{i+n}}%
(W_{i+n}-W_{i})^{\top}\big]+R_{Z,i}.\nonumber
\end{equation}
Hence, the time-discretization of $Z$ is, for $i=N-m,N-m-1,\cdots,0$
\begin{equation}
Z_{i}^{\pi}= \sum_{n=1}^{m}\lambda_{m,n}\mathbb{E}_{i}\big[Y_{i+n}^{\pi
}(W_{i+n}-W_{i})^{\top}\big].\label{2-30}%
\end{equation}

Let $\widetilde{Y}_{i}^{\pi}$ denote the approximation to $u(t_{i},X^{\pi}%
_{i})$ via the predictor part. Set the improved approximation $Y_{i}^{\pi}$
found in the corrector part. $\widetilde{\beta_{j}}$ replaces the value of
$\beta_{j}$ in the Adams-Bashforth formula while $\beta_{j}$ denotes the
Adams-Moulton coefficients. Correspondingly, the parameter $k$ denotes in the
Adams-Moulton formula and the Adams-Bashforth formula can be denoted by
$\widetilde{k}$. Hence, the predictor-corrector scheme based on the Adams
types could be expressed as below, for $i=N-\max(\widetilde{k},k),\cdots
,1,0:$
\begin{equation}
\left\{
\begin{array}
[c]{rl}%
\widetilde{Y}^{\pi}_{i} = & \mathbb{E}_{i}[Y^{\pi}_{i+1} +h\sum\limits_{j=1}%
^{\widetilde{k}}\widetilde{\beta}_{j}f^{\pi}_{i+j}],\\
Y_{i}^{\pi}= & \mathbb{E}_{i}[Y_{i+1}^{\pi}+ h \beta_{0}\widetilde{f}_{i}%
^{\pi}+ h\sum\limits_{j=1}^{k}\beta_{j}f^{\pi}_{i+j}],\\
Z_{i}^{\pi}= & \sum\limits_{n=1}^{\max(\widetilde{k},k)}\lambda_{\max
(\widetilde{k},k),n}\mathbb{E}_{i}\big[Y_{i+n}^{\pi}(W_{i+n}-W_{i})^{\top
}\big],
\end{array}
\right.  \label{2-15}%
\end{equation}
where $\widetilde{\beta}_{1},\widetilde{\beta}_{2},\cdots,\widetilde{\beta
}_{\widetilde{k}}$ and $\beta_{0},\beta_{1},\beta_{2},\cdots,\beta_{k}$ are
constants and would be given in the following. This scheme is implemented by
means of Adams types i.e. Adams-Bashforth method is adopted by a preliminary
computation. Subsequently, this numerical solution is used in the
Adams-Moulton formula to yield the derivative value at the new point. The
original idea of this scheme is extending the Euler method via allowing the
numerical solution to depend on several previous step values of solutions and
derivatives (see \cite{2,31,32,33,34} for detail about ODEs and \cite{36}
w.r.t. SDEs). The scheme (\ref{2-15}) is referred to as the
predictor-corrector method because the total calculation in a step is made up
of a preliminary prediction of the numerical solution and followed by a
correction of this predicted answer.

Usually, the coefficients $k$ and $\widetilde{k}$ can take different values.
To obtain the same order of local truncation error, the coefficients $k$ and
$\widetilde{k}$ have the relation $\widetilde{k}= k+1$. In addition, the
scheme (\ref{2-15}) can be rewritten as, for $i=N-k-1,\cdots,1,0:$
\begin{equation}
\left\{
\begin{array}
[c]{rl}%
\widetilde{Y}_{i}^{\pi}= & \mathbb{E}_{i}[Y_{i+1}^{\pi}+h\sum\limits_{j=1}%
^{k+1}\widetilde{\beta}_{j}f^{\pi}_{i+j}],\\
Y^{\pi}_{i} = & \mathbb{E}_{i}[Y^{\pi}_{i+1} + h \beta_{0}\widetilde{f}^{\pi
}_{i} + h\sum\limits_{j=1}^{k}\beta_{j}f^{\pi}_{i+j}],\\
Z^{\pi}_{i} = & \sum\limits_{n=1}^{k+1}\lambda_{k+1,n}\mathbb{E}%
_{i}\big[Y_{i+n}^{\pi}(W_{i+n}-W_{i})^{\top}\big].
\end{array}
\right.  \label{2-16}%
\end{equation}

The scheme (\ref{2-15}) provides an algorithm for calculating $(Y^{\pi
}_{N-k-1},Z^{\pi}_{N-k-1})$ in terms of $(Y^{\pi}_{N},Z^{\pi}_{N}),$ $(Y^{\pi
}_{N-1},Z^{\pi}_{N-1}),$ $\cdots,$ $(Y^{\pi}_{N-k},Z^{\pi}_{N-k})$. The
subsequent approximation solutions can be found via the same manner. However,
one has to consider how to obtain the value of $(Y^{\pi}_{N-1},Z^{\pi}%
_{N-1}),$ $(Y^{\pi}_{N-2},Z^{\pi}_{N-2}),\cdots,(Y^{\pi}_{N-k},Z^{\pi}_{N-k}%
)$. Of course, it is possible to evaluate $(Y^{\pi}_{N-1},Z^{\pi}%
_{N-1}),(Y^{\pi}_{N-2},Z^{\pi}_{N-2}),\cdots,(Y^{\pi}_{N-k},Z^{\pi}_{N-k})$
via a low order method, such as Euler scheme. Nevertheless, this maybe
introduce much bigger errors and lead to nullification of the advantages of
the subsequent use of the high order scheme. For this difficulty, we can
utilize the Runge-Kutta scheme which is presented by J.-F. Chassagneux and D.
Crisan \cite{10} or the scheme (\ref{2-15}) with $\widetilde{k}=1,k=0$ with a
smaller time step (see \cite{38} for details).

In what follows, before providing the parameters in scheme (\ref{2-16}), we
first give the following definition.

\begin{definition}
Suppose that $\left( u(t,X_{t}),\nabla_{x}u(t,X_{t})\sigma(t,X_{t})\right) $
is the exact solution of the BSDE in (\ref{2-1}). Let the local truncation
error with respect to $Y$ be
\begin{equation}
T_{i} = u(t_{i},X_{i}^{\pi}) - Y_{i}^{\pi},\nonumber
\end{equation}
where $Y_{i}^{\pi}$ denotes the numerical solution of the BSDE in (\ref{2-1}).
Furthermore, the multi-step scheme (\ref{2-16}) with respect to $Y$ is said to
have $n$-order accuracy ($n\in\mathbb{N}^{+}$) if the local truncation error
$T_{i}$ satisfies $T_{i} = O(h^{n+1})$.
\end{definition}

From Lemma 2.1, the integrand $\mathbb{E}_{t}[f(s,X_{s},Y_{s},Z_{s})],s>t$ is
a continuous function w.r.t. $s$. Then, by taking derivative w.r.t. $s$ on
\[
\mathbb{E}_{t}[Y_{s}]= \mathbb{E}_{t}[\Phi(X_{T})] + \int_{s}^{T}
\mathbb{E}_{t}[f(\bar{s},X_{\bar{s}},Y_{\bar{s}},Z_{\bar{s}})]d\bar{s},
\quad\forall s\in[t,T],
\]
we obtain the following reference ordinary differential equation
\begin{equation}
\frac{d\mathbb{E}_{t}[Y_{s}]}{ds}=-\mathbb{E}_{t}[f(s,X_{s},Y_{s}%
,Z_{s})],~~~s\in\lbrack t,T]. \label{2-12}%
\end{equation}

Assume that no errors have yet been introduced when the approximation at
$(t_{i},X_{i})$ is about to be calculated. By (\ref{2-12}), we get
$\frac{d\mathbb{E}_{i}[Y^{\pi}_{i+j}]}{dt}= -\mathbb{E}_{i}[f^{\pi}%
_{i+j}]=\mathbb{E}_{i}[u^{(0)}(t_{i+j},X^{\pi}_{i+j})],j=0,1,2,\cdots$. Thus,
\begin{align}
T_{i}= &  \mathbb{E}_{i}\Big[u(t_{i},X^{\pi}_{i}) - u(t_{i+1},X^{\pi}_{i+1}) -
h \sum_{\ell=0}^{k}\beta_{\ell}f^{\pi}_{i+\ell}\Big]\nonumber\\
=  &  \mathbb{E}_{i}\Big[u(t_{i},X^{\pi}_{i}) - u(t_{i+1},X^{\pi}_{i+1}) +
h\sum_{\ell=0}^{k}\beta_{\ell}u^{(0)}(t_{i+\ell},X^{\pi}_{i+\ell
})\Big]\nonumber\\
= &  \mathbb{E}_{i}\Big[hu^{(0)}(t_{i},X^{\pi}_{i})(-1+\beta_{0}+\beta
_{1}+\beta_{2}+\cdots+\beta_{k})\nonumber\\
&  +h^{2}u^{(0,0)}(t_{i},X^{\pi}_{i})(-\frac{1}{2}+\beta_{1}+2\beta_{2}%
+\cdots+k\beta_{k})\nonumber\\
&  +h^{3}u^{(0,0,0)}(t_{i},X^{\pi}_{i})\big(-\frac{1}{6}+\frac{1}{2}(\beta
_{1}+2^{2}\beta_{2}+\cdots+k^{2}\beta_{k})\big)\nonumber\\
&  +\cdots\nonumber\\
&  +h^{k}u^{(0)_{k}}(t_{i},X^{\pi}_{i})\big(-\frac{1}{k!}+\frac{1}%
{(k-1)!}(\beta_{1}+2^{k-1}\beta_{2}+\cdots+k^{k-1}\beta_{k})\big)\Big].
\label{2-13}%
\end{align}
Then $T_{i}$ has an expression as below via the equation (\ref{2-13})
\begin{equation}
C_{0}u(t_{i},X^{\pi}_{i}) + C_{1}hu^{(0)}(t_{i},X^{\pi}_{i})+C_{2}%
h^{2}u^{(0,0)}(t_{i},X^{\pi}_{i})+ \cdots+C_{k}h^{k}u^{(0)_{k}}(t_{i},X^{\pi
}_{i})+ O(h^{k+1}). \label{2-14}%
\end{equation}
If $C_{0} = C_{1}=\cdots=C_{k} =0, C_{k+1}\neq0$, then the local truncation
error can be estimated as $O(h^{k+1})$. Now, the method has order $k$. In
Table 1, we provides the value of parameters for $k=1,2,3,4,5,6$ (for
$k=1,2,3,4$ see the Table in page 16 of \cite{JFC13}). \begin{table}[h]
\caption{ coefficients for predictor-corrector scheme based on Adams type}%
\centering\resizebox{\textwidth}{40mm}{
\begin{tabular}
[c]{|cc|cc|ccccccc|c|}\hline
\multicolumn{2}{|c|}{order} & \multicolumn{2}{c|}{term} & $\beta_{0}$ &
$\beta_{1}$ & $\beta_{2}$ & $\beta_{3}$ & $\beta_{4}$ & $\beta_{5}$ &
$\beta_{6}$ & error~~constant\\\hline
\multicolumn{2}{|c|}{{1}} & \multicolumn{2}{c|}{predictor} & 0 & 1 &  &  &  &
&  & $\frac{1}{2}$\\
\multicolumn{2}{|c|}{} & \multicolumn{2}{c|}{corrector} & 1 & 0 &  &  &  &  &
& $-\frac{1}{2}$\\\hline
\multicolumn{2}{|c|}{{2}} & \multicolumn{2}{c|}{predictor} & 0 & $\frac{3}{2}$
& $-\frac{1}{2}$ &  &  &  &  & $-\frac{5}{12}$\\
\multicolumn{2}{|c|}{} & \multicolumn{2}{c|}{corrector} & $\frac{1}{2}$ &
$\frac{1}{2}$ &  &  &  &  &  & $\frac{1}{12}$\\\hline
\multicolumn{2}{|c|}{{3}} & \multicolumn{2}{c|}{predictor} & 0 & $\frac
{23}{12}$ & $-\frac{4}{3}$ & $\frac{5}{12}$ &  &  &  & $\frac{3}{8}$\\
\multicolumn{2}{|c|}{} & \multicolumn{2}{c|}{corrector} & $\frac{5}{12}$ &
$\frac{2}{3}$ & $-\frac{1}{12}$ &  &  &  &  & $-\frac{1}{24}$\\\hline
\multicolumn{2}{|c|}{{4}} & \multicolumn{2}{c|}{predictor} & 0 & $\frac
{55}{24}$ & $-\frac{59}{24}$ & $\frac{37}{24}$ & $-\frac{3}{8}$ &  &  &
$-\frac{251}{720}$\\
\multicolumn{2}{|c|}{} & \multicolumn{2}{c|}{corrector} & $\frac{3}{8}$ &
$\frac{19}{24}$ & $-\frac{5}{24}$ & $\frac{1}{24}$ &  &  &  & $\frac{19}{720}$\\\hline
\multicolumn{2}{|c|}{{5}} & \multicolumn{2}{c|}{predictor} & 0 & $\frac
{1901}{720}$ & $-\frac{1387}{360}$ & $\frac{109}{30}$ & $-\frac{637}{360}$ &
$\frac{251}{720}$ &  & $\frac{95}{288}$\\
\multicolumn{2}{|c|}{} & \multicolumn{2}{c|}{corrector} & $\frac{251}{720}$ &
$\frac{323}{360}$ & $-\frac{11}{30}$ & $\frac{53}{360}$ & $-\frac{19}{720}$ &
&  & $-\frac{3}{160}$\\\hline
\multicolumn{2}{|c|}{{6}} & \multicolumn{2}{c|}{predictor} & 0 & $\frac
{4277}{1440}$ & $-\frac{2641}{480}$ & $\frac{4991}{720}$ & $-\frac{3649}{720}$
& $\frac{959}{480}$ & $-\frac{95}{288}$ & $-\frac{19087}{60480}$\\
\multicolumn{2}{|c|}{} & \multicolumn{2}{c|}{corrector} & $\frac{95}{288}$ &
$\frac{1427}{1440}$ & $-\frac{133}{240}$ & $\frac{241}{720}$ & $-\frac
{173}{1440}$ & $\frac{3}{160}$ &  & $\frac{863}{60480}$\\\hline
\end{tabular}}\end{table}

\section{Main results}

In this part, we introduce the predictor-corrector type general linear
multi-step schemes of decoupled FBSDEs in detail and investigate the
corresponding stability and convergence.

\subsection{Predictor-corrector discretization via the general linear
multi-step scheme}

In this subsection, we extend linear multi-step schemes (\cite{9,JFC13}) to
the predictor-corrector type general linear multi-step schemes.

Our aim is to deduce the discretization of BSDE backward in time based on the
general linear multi-step scheme if $\{Y_{l}^{\pi}\}_{N-m+1\le l\le N}$ and
$\{Z_{l}^{\pi}\}_{N-m+1\le l\le N}$ are available. Namely, for
$i=N-m,N-m-1,\cdots,0$
\begin{equation}
Y_{i}^{\pi}= \mathbb{E}_{i}\Big[\sum_{j=1}^{m} \alpha_{j}Y_{i+j}^{\pi}+
\sum_{j=1}^{m} \gamma_{j}hf_{i+j}^{\pi}\Big], \label{2-20}%
\end{equation}
or as
\begin{equation}
Y_{i}^{\pi}= \mathbb{E}_{i}\Big[\sum_{j=1}^{m} \alpha_{j}Y_{i+j}^{\pi}+
\gamma_{0}hf_{i}^{\pi}+ \sum_{j=1}^{m} \gamma_{j}hf_{i+j}^{\pi}\Big],
\label{2-21}%
\end{equation}
where $\{\alpha_{l}\}_{1\le l\le m}$ and $\{\gamma_{l}\}_{1\le l\le m}$ are
real numbers. In particular, let $\gamma_{0}\neq0$ be a real number. Now,
(\ref{2-20}) is an explicit scheme with respect to $Y$, while (\ref{2-21}) is
an implicit scheme.

As for the time-discretization of the term $Z$, we adopt the scheme presented
in the subsection 2.2. Thus, the equations (\ref{2-21}) and (\ref{2-30})
consist of a discrete-time approximation $(Y_{i}^{\pi},Z_{i}^{\pi})$ for
$(Y,Z)$ at $t_{i}$: for $i=N$
\begin{align}
Y_{N}^{\pi}= \Phi(X_{N}^{\pi}),~~~Z_{N}^{\pi}=\sigma(t_{N},X_{N}^{\pi}%
)D_{x}\Phi(X_{N}^{\pi}).\nonumber
\end{align}
For $i=N-1,N-2,\cdots,N-m+1$, an appropriate one-step scheme can be utilized
to solve the BSDE. For example, we can adjust the parameters of the scheme
(\ref{2-15}) such that it becomes one-step scheme and satisfies the required
accuracy by using a smaller time step. For $i= N-m,N-m-1,\cdots,1,0$
\begin{equation}
\left\{
\begin{array}
[c]{rl}%
Y_{i}^{\pi}= & \mathbb{E}_{i}\Big[\sum\limits_{j=1}^{m} \alpha_{j}Y_{i+j}%
^{\pi}+ \gamma_{0}hf_{i}^{\pi}+ \sum\limits_{j=1}^{m} \gamma_{j}hf_{i+j}^{\pi
}\Big],\\
Z_{i}^{\pi}= & \sum\limits_{n=1}^{m}\lambda_{m,n}\mathbb{E}_{i}\big[Y_{i+n}%
^{\pi}(W_{i+n}-W_{i})^{\top}\big].
\end{array}
\right.  \label{2-31}%
\end{equation}
This scheme is explicit w.r.t. $Z$ and implicit w.r.t. $Y$. Of course, we can
calculate the numerical solutions of BSDE via (\ref{2-31}). But in general,
the implicit scheme requires an algebraic equation to be solved at each time
step. This imposes an additional computational burden. For this difficulty, we
introduce the predictor-corrector method. The general linear multi-step
predictor-corrector method is constructed as below:
\begin{equation}
\left\{
\begin{array}
[c]{rl}%
\widetilde{Y}_{i}^{\pi}= & \mathbb{E}_{i}\Big[\sum\limits_{j=1}^{m}
\widetilde{\alpha}_{j}Y_{i+j}^{\pi}+ \sum\limits_{j=1}^{m} \widetilde{\gamma
}_{j}hf_{i+j}^{\pi}\Big],\\
Y_{i}^{\pi}= & \mathbb{E}_{i}\Big[\sum\limits_{j=1}^{m} \alpha_{j}Y_{i+j}%
^{\pi}+ \gamma_{0}h\widetilde{f}_{i}^{\pi}+ \sum\limits_{j=1}^{m} \gamma
_{j}hf_{i+j}^{\pi}\Big],\\
Z_{i}^{\pi}= & \mathbb{E}_{i}\big[\sum\limits_{j=1}^{m}\lambda_{m,j}%
Y_{i+j}^{\pi}(W_{i+j}-W_{i})^{\top}\big],
\end{array}
\right.  \label{2-32}%
\end{equation}
where $\{\widetilde{\alpha}_{l}\}_{1\le l\le m}$ and $\{\widetilde{\gamma}%
_{l}\}_{1\le l\le m}$ are real numbers. At the $i$-th time step, the predictor
is constructed by using an explicit general linear multi-step scheme which
predicts a value of $Y$ denoted by $\widetilde{Y}_{i}^{\pi}$. Then the
corrector whose structure is similar to an implicit general linear multi-step
scheme is applied to correct the predicted value. We emphasize that not only
the predictor step is explicit, but also the corrector step is explicit.

Next, we provide two schemes which are the variant forms of the scheme
(\ref{2-32}). In other words, these schemes are the special cases of
(\ref{2-32}). If the predictor term $\widetilde{Y}$ is calculated via the
Adams-Bashforth method, the scheme (\ref{2-32}) can be restated as below:
\begin{equation}
\left\{
\begin{array}
[c]{rl}%
\widetilde{Y}^{\pi}_{i} = & \mathbb{E}_{i}[Y^{\pi}_{i+1} +h\sum\limits_{j=1}%
^{\widetilde{m}}\widetilde{\beta}_{j}f^{\pi}_{i+j}],\\
Y_{i}^{\pi}= & \mathbb{E}_{i}\Big[\sum\limits_{j=1}^{m} \alpha_{j}Y_{i+j}%
^{\pi}+ \gamma_{0}h\widetilde{f}_{i}^{\pi}+ \sum\limits_{j=1}^{m} \gamma
_{j}hf_{i+j}^{\pi}\Big],\\
Z_{i}^{\pi}= & \mathbb{E}_{i}\big[\sum\limits_{j=1}^{m}\lambda_{m,j}%
Y_{i+j}^{\pi}(W_{i+j}-W_{i})^{\top}\big].
\end{array}
\right.  \label{2-33}%
\end{equation}
We can also naturally derive the following linear multi-step scheme by
changing the calculation expression of $Z$ (see \cite{20}).
\begin{equation}
\left\{
\begin{array}
[c]{rl}%
\widetilde{Y}_{i}^{\pi}= & \mathbb{E}_{i}\Big[\sum\limits_{j=1}^{m}
\widetilde{\alpha}_{j}Y_{i+j}^{\pi}+ \sum\limits_{j=1}^{m} \widetilde{\gamma
}_{j}hf_{i+j}^{\pi}\Big],\\
Y_{i}^{\pi}= & \mathbb{E}_{i}\Big[\sum\limits_{j=1}^{m} \alpha_{j}Y_{i+j}%
^{\pi}+ \gamma_{0}h\widetilde{f}_{i}^{\pi}+ \sum\limits_{j=1}^{m} \gamma
_{j}hf_{i+j}^{\pi}\Big],\\
Z_{i}^{\pi}= & \mathbb{E}_{i}\Big[\big(\sum\limits_{j=1}^{m} \alpha
_{j}Y_{i+1+j}^{\pi}+ \gamma_{0}h\widetilde{f}_{i+1}^{\pi}+ \sum\limits_{j=1}%
^{m} \gamma_{j}hf_{i+1+j}^{\pi}\big)\frac{\Delta W_{i}^{\top}}{h}\Big].
\end{array}
\right.  \label{2-34}%
\end{equation}

In what follows, our goal is to investigate the relation of the parameters
$\alpha_{j}$ and $\gamma_{j}$ under the conditions of stability and high order
rate of convergence. This is necessary for the reason that we cannot implement
the scheme (\ref{2-32}) to calculate BSDEs if the parameters $\alpha_{j}$ and
$\gamma_{j}$ are not known. Combined (\ref{2-12}), (\ref{2-14}) with
It\^{o}-Taylor expansion, the local truncation error $\widetilde{T}_{i}$ of
scheme (\ref{2-32}) w.r.t. $Y$ is, for $k=m$
\begin{align}
\widetilde{T}_{i}= & \mathbb{E}_{i}\Big[u(t_{i},X_{i}^{\pi})-\sum_{\ell=1}%
^{m}\alpha_{\ell}u(t_{i+\ell},X_{i+\ell}^{\pi})-h\sum_{\ell=0}^{m}\gamma
_{\ell}f_{i+\ell}^{\pi}\Big]\nonumber\\
& =\mathbb{E}_{i}\Big[u(t_{i},X_{i}^{\pi})-\sum_{\ell=1}^{m}\alpha_{\ell
}u(t_{i+\ell},X_{i+\ell}^{\pi})+h\sum_{\ell=0}^{m}\gamma_{\ell}u^{(0)}%
(t_{i+\ell},X_{i+\ell}^{\pi})\Big]\nonumber\\
& =\mathbb{E}_{i}\Big[u(t_{i},X_{i}^{\pi})(1-\sum_{\ell=1}^{m}\alpha_{\ell
})+hu^{(0)}(t_{i},X_{i}^{\pi})(-\sum_{\ell=1}^{m}\ell\alpha_{\ell}+\sum
_{\ell=0}^{m}\gamma_{\ell})\nonumber\\
& ~~+h^{2}u^{(0,0)}(t_{i},X_{i}^{\pi})(-\frac{1}{2}\sum_{\ell=1}^{m}\ell
^{2}\alpha_{\ell}+\sum_{\ell=1}^{m}\ell\gamma_{\ell})+h^{3}u^{(0,0,0)}%
(t_{i},X_{i}^{\pi})\big(-\frac{1}{6}\sum_{\ell=1}^{m}\ell^{3}\alpha_{\ell
}+\frac{1}{2}\sum_{\ell=1}^{m}\ell^{2}\gamma_{\ell}\big)\nonumber\\
& ~~+\cdots\nonumber\\
& ~~+h^{m}u^{(0)_{m}}(t_{i},X_{i}^{\pi})\big(-\frac{1}{m!}\sum_{\ell=1}%
^{m}\ell^{m}\alpha_{\ell}+\frac{1}{(m-1)!}\sum_{\ell=1}^{m}\ell^{m-1}%
\gamma_{\ell}\big)\Big]+O(h^{m+1}). \label{2-35}%
\end{align}
Set
\begin{equation}
\left\{
\begin{array}
[c]{rl}%
C_{0}= & 1-\sum\limits_{\ell=1}^{m}\alpha_{\ell},\\
C_{1}= & -\sum\limits_{\ell=1}^{m}\ell\alpha_{\ell}+\sum_{\ell=0}^{m}%
\gamma_{\ell},\\
C_{2}= & -\frac{1}{2}\sum\limits_{\ell=1}^{m}\ell^{2}\alpha_{\ell}%
+\sum\limits_{\ell=1}^{m}\ell\gamma_{\ell},\\
\cdots & \\
C_{m}= & -\frac{1}{m!}\sum\limits_{\ell=1}^{m}\ell^{m}\alpha_{\ell}+\frac
{1}{(m-1)!}\sum\limits_{\ell=1}^{m}\ell^{m-1}\gamma_{\ell}.
\end{array}
\right.  \label{2-36}%
\end{equation}
If $C_{0}=C_{1}=C_{2}=\cdots=C_{m}=0$ and $C_{m+1}\neq0$, then the local
truncation error accuracy of scheme (\ref{2-32}) reaches $m$-order.

\begin{remark}
The (\ref{2-36}) implies that we could obtain a family of schemes reaching
$m$-order because the number of unknowns are greater than those of equations.
This is the main difference from the scheme (\ref{2-15}). Moreover, it
indicates that the scheme (\ref{2-15}) is a special form of the scheme of
(\ref{2-32}).
\end{remark}

\subsection{Error estimates of the scheme (\ref{2-32})}

In this subsection, we concentrate on exploring the stability and high order
accuracy of the scheme (\ref{2-32}). Before demonstrating them, we first
present a necessary property, a lemma and two definitions.

\begin{proposition}
\label{pro-pc} Assume that $f_{i}^{\pi}$ is smooth enough and $\widetilde{k}=
k +1$ in scheme (\ref{2-15}). For $i< N-k, k\in[0,N)$, it follows that
\begin{equation}
|u(t_{i},X^{\pi}_{i})- Y_{i}^{\pi}| = |\frac{C_{k+2}}{C_{k+2}-\widetilde{C}%
_{k+2}}||\widetilde{Y}^{\pi}_{i} - Y^{\pi}_{i} |, \label{3-1}%
\end{equation}
where $\widetilde{C}_{k+2}$ denotes the error constant for the predictor
$(k+1)$-order term and $C_{k+2}$ denotes the error constant for $(k+1)$-order
corrector term.
\end{proposition}

\begin{proof}
It is straightforward that there exist two approximations to the exact
solution $u(t_{i},X^{\pi}_{i})$ in every step in scheme (\ref{2-15}).
Moreover, the predictor term and the corrector term possess different error
constants even though both of them have the same order. Thus, the error in the
predictor term is equal to
\begin{equation}
u(t_{i},X^{\pi}_{i})= \widetilde{Y}^{\pi}_{i} + h^{k+2}\widetilde{C}%
_{k+2}u^{(0)_{k+2}}(t_{i},X^{\pi}_{i}) + o(h^{k+2}). \label{3-2}%
\end{equation}
Similarly, we can obtain the error of the corrector term at the time step $i$
\begin{equation}
u(t_{i},X^{\pi}_{i})= Y^{\pi}_{i} + h^{k+2}C_{k+2}u^{(0)_{k+2}}(t_{i},X^{\pi
}_{i}) + o(h^{k+2}). \label{3-3}%
\end{equation}
Subtracting (\ref{3-2}) from (\ref{3-3}) and ignoring higher order term, one
has
\begin{equation}
u^{(0)_{k+2}}(t_{i},X^{\pi}_{i}) = \frac{1}{h^{k+2}(C_{k+2}-\widetilde{C}%
_{k+2})}(\widetilde{Y}^{\pi}_{i} - Y^{\pi}_{i}). \label{3-4}%
\end{equation}
Plugging (\ref{3-4}) into (\ref{3-3}) and neglecting higher order term, we
obtain
\begin{equation}
u(t_{i},X^{\pi}_{i})- Y_{i}^{\pi}= \frac{C_{k+2}}{C_{k+2}-\widetilde{C}_{k+2}%
}(\widetilde{Y}^{\pi}_{i} - Y^{\pi}_{i} ).\nonumber
\end{equation}
The proof is completed.
\end{proof}

Next, we provide a lemma and two definitions which will be used to deduce the
stability and high order accuracy of the scheme (\ref{2-32}),

\begin{lemma}
(see Lemma 3 in \cite{39}) \label{est-1} Suppose that $N$ and $K$ are two
nonnegative integers with $N\ge K$ and $h$ any positive number. Let
$\{\eta_{i}\}$ be a series satisfying
\begin{equation}
|\eta_{i}|\le\beta+\alpha h\sum_{j=i+1}^{N}|\eta_{j}|,\quad i =
N-K,N-K-1,\cdots,0,\nonumber
\end{equation}
where $\alpha$ and $\beta$ are two positive constants. Let $M_{0} =
\max\limits_{N-K\le j\le N}|\eta_{j}|$ and $T = Nh$; then
\begin{equation}
|\eta_{i}|\le exp(\alpha T)(\beta+\alpha KhM_{0}),\quad i = N-K,N-K-1,\cdots
,0.\nonumber
\end{equation}

\end{lemma}

\begin{definition}
The characteristic polynomials of (\ref{2-21}) are given by
\begin{equation}
P(\zeta) = \zeta^{m}- \alpha_{1}\zeta^{m-1} - \alpha_{2}\zeta^{m-2}%
-\cdots-\alpha_{m}. \label{4-1}%
\end{equation}
The equation (\ref{2-21}) is said to fulfil Dahlquist's root condition, if

i) The roots of $P(\zeta)$ lie on or within the unit circle;

ii) The roots on the unit circle are simple.
\end{definition}

\begin{definition}
Let $(Y_{i}^{\pi},Z_{i}^{\pi}),i= 0,1,\cdots,N-m$ be the time-discretization
approximate solution given by the scheme (\ref{2-32}) and $(\bar{Y}_{i}^{\pi
},\bar{Z}_{i}^{\pi})$ is the solution of its perturbed form (see (\ref{4-4})
). Then the scheme (\ref{2-32}) is said to be $\mathbb{L}_{2}$-stable if
\begin{equation}
\max\limits_{0\le i \le N-m}\mathbb{E}[|Y_{i}^{\pi}-\bar{Y}_{i}^{\pi}|^{2}] +
\sum\limits_{i=0}^{N-m}h \mathbb{E}[|Z_{i}^{\pi}-\bar{Z}_{i}^{\pi}|^{2}] \le C
\left( \max\limits_{N-m+1\le k \le N}|Y_{k}^{\pi}-\bar{Y}_{k}^{\pi}|^{2}
+\sum_{i=0}^{N-m}\mathbb{E}_{i}\big[h|\varepsilon_{i}^{Z}|^{2}+\frac{1}%
{h}|\varepsilon_{i}^{Y}|^{2}\big]\right) , \label{4-6}%
\end{equation}
where $C$ is a constant; $(\bar{Y}_{i}^{\pi},\bar{Z}_{i}^{\pi})$ satisfies a
perturbed form of (\ref{2-32}) for $i = N-m,N-m-1,\cdots,0$
\begin{equation}
\left\{
\begin{array}
[c]{rl}%
\bar{\widetilde{Y}}_{i}^{\pi}= & \mathbb{E}_{i}\Big[\sum\limits_{j=1}^{m}
\widetilde{\alpha}_{j}\bar{Y}_{i+j}^{\pi}+ \sum\limits_{j=1}^{m}
\widetilde{\gamma}_{j}hf(t_{i+j},X^{\pi}_{i+j},\bar{Y}_{i+j}^{\pi},\bar
{Z}_{i+j}^{\pi})\Big],\\
\bar{Y}_{i}^{\pi}= & \mathbb{E}_{i}\Big[\sum\limits_{j=1}^{m} \alpha_{j}%
\bar{Y}_{i+j}^{\pi}+ \gamma_{0}hf(t_{i+j},X^{\pi}_{i+j},\bar{\widetilde{Y}%
}_{i}^{\pi},\bar{Z}_{i}^{\pi}) + \sum\limits_{j=1}^{m} \gamma_{j}%
hf(t_{i+j},X^{\pi}_{i+j},\bar{Y}_{i+j}^{\pi},\bar{Z}_{i+j}^{\pi}%
)\Big]+\varepsilon_{i}^{Y},\\
\bar{Z}_{i}^{\pi}= & \mathbb{E}_{i}\big[\sum\limits_{j=1}^{m}\lambda_{m,j}%
\bar{Y}_{i+j}^{\pi}(W_{i+j}-W_{i})^{\top}\big]+\varepsilon_{i}^{Z}.
\end{array}
\right.  \label{4-4}%
\end{equation}
Sequences $\varepsilon_{i}^{Y}$ and $\varepsilon_{i}^{Z}$ which belong to
$\mathbb{L}_{2}(\mathcal{F}_{i})$ are random variables.
\end{definition}

Note that we are merely interested in the solution of the BSDE in (\ref{2-1}).
Therefore, we assume that the solution of SDE in (\ref{2-1}) can be obtained
perfectly. Thus, we do not consider the error caused by $X_{t}$ (see \cite{38}).

\begin{theorem}
Suppose \textbf{Assumption 2 (i)} and \textbf{Assumption 2 (iii)} hold. Then
the stochastic multi-step method is numerically stable if and only if its
characteristic polynomial (\ref{4-1}) satisfies Dahlquist's root condition.
\end{theorem}

\begin{proof}
\textbf{Sufficiency:} Let $\Delta Y_{i}=Y_{i}^{\pi}-\bar{Y}_{i}^{\pi},\Delta
Z_{i}=Z_{i}^{\pi}-\bar{Z}_{i}^{\pi},\Delta f_{i}= f(t_{i},X_{i}^{\pi}%
,Y_{i}^{\pi},Z_{i}^{\pi})-f(t_{i},X_{i}^{\pi},\bar{Y}_{i}^{\pi},\bar{Z}%
_{i}^{\pi})$ for $i=N-m,N-m-1,\cdots,0.$ We complete the proof of the theorem
in three steps.\newline\textbf{step 1.} From (\ref{2-32}) and (\ref{4-4})
w.r.t. $Y$, one obtains
\[
\Delta Y_{i} = \mathbb{E}_{i}\Big[\sum_{j=1}^{m} \alpha_{j}\Delta Y_{i+j} +
\gamma_{0}h\Delta\widehat{f}_{i} + \sum_{j=1}^{m} \gamma_{j}h\Delta
f_{i+j}\Big]- \varepsilon_{i}^{Y},
\]
where $\Delta\widehat{f}_{i}=f(t_{i},X_{i}^{\pi},\widetilde{Y}_{i}^{\pi}%
,Z_{i}^{\pi})-f(t_{i},X^{\pi}_{i},\bar{\widetilde{Y}}_{i}^{\pi},\bar{Z}%
_{i}^{\pi})$. We rearrange the $m$-step recursion to a one-step recursion as
follow
\begin{equation}%
\begin{array}
[c]{rl}%
\mathbb{E}_{i}[\mathcal{Y}_{i}] = & \mathbb{E}_{i}[A\mathcal{Y}_{i+1} + F_{i}
+ R_{i} ],
\end{array}
\label{4-7}%
\end{equation}
where
\[
\mathcal{Y}_{i}=%
\begin{pmatrix}
\Delta Y_{i}\\
\Delta Y_{i+1}\\
\vdots\\
\Delta Y_{i+m-1}%
\end{pmatrix}
, A=%
\begin{pmatrix}
\alpha_{1} & \alpha_{2} & \cdots & \alpha_{m}\\
1 & 0 &  & \\
& \ddots & \ddots & \\
&  & 1 & 0
\end{pmatrix}
, F_{i}=%
\begin{pmatrix}
\gamma_{0}h\Delta\widehat{f}_{i}+\sum\limits_{j=1}^{m} \gamma_{j}h\Delta
f_{i+j}\\
0\\
\vdots\\
0
\end{pmatrix}
, R_{i}=
\begin{pmatrix}
-\varepsilon_{i}^{Y}\\
0\\
\vdots\\
0
\end{pmatrix}
.
\]
To ensure the stability of the $m$-step scheme, the norm of the matrix $A$ in
the equation (\ref{4-7}) is no more than 1 (see \cite{23}, Chapter III.4,
Lemma 4.4). This can be satisfied if the eigenvalues $eig(A)$ of the matrix
$A$ make $|eig(A)|\le1$ and in which the eigenvalues are simple if
$|eig(A)|=1$. In addition, the eigenvalues of $A$ satisfy the root condition
by Definition 3.4. By the Dahlquist's root condition, it is possible that
there exists a non-singular matrix $\mathcal{D}$ such that $||\mathcal{D}%
^{-1}A\mathcal{D}||_{2}\le1$ where $||\cdot||_{2}$ denotes the spectral matrix
norm induced by Euclidian vector norm in $\mathbb{R}^{m\times n}$. Hence, we
can choose a scalar product for $\bar{A},\widetilde{A }\in\mathbb{R}^{m\times
n}$ as $\langle\bar{A},\widetilde{A}\rangle_{*}:=\langle\mathcal{D}^{-1}%
\bar{A},\mathcal{D}^{-1}\widetilde{A}\rangle= \bar{A}^{\top}(\mathcal{D}%
^{-1})^{\top}\mathcal{D}^{-1}\widetilde{A}$. And we have $|\bar{A}|_{*}%
^{2}:=\langle\bar{A},\bar{A}\rangle_{*}$ with the induced vector norm on
$\mathbb{R}^{m\times n}$. Let $||A||_{*}=||\mathcal{D}^{-1}A\mathcal{D}||_{2}$
be the induced matrix norm. Owing to the norm equivalence, we know that there
exist positive constants $c_{1},c_{2}$ such that
\begin{equation}
c_{1}|\bar{A}|_{2}^{2}\le|\bar{A}|_{*}^{2}\le c_{2}|\bar{A}|_{2}^{2},
~~\forall\bar{A}\in\mathbb{R}^{m\times n} \label{4-8}%
\end{equation}
where $|\bar{A}|_{2}^{2} = \sum\limits_{j=1,2,\cdots,m}|a_{j}|^{2}$ for
$\bar{A}=(a_{1}^{\top},\cdots,a_{m}^{\top})^{\top}$. Applying $|\cdot|_{\ast}$
to the equation (\ref{4-7}), we have
\begin{equation}%
\begin{array}
[c]{rl}%
|\mathbb{E}_{i}[\mathcal{Y}_{i}]|_{\ast}= & \Big|\mathbb{E}_{i}[A\mathcal{Y}%
_{i+1}+F_{i}+R_{i}]\Big|_{\ast}\\
= & ||A||_{\ast}|\mathbb{E}_{i}[\mathcal{Y}_{i+1}]|_{\ast}+|\mathbb{E}%
_{i}[F_{i}]|_{\ast}+|\mathbb{E}_{i}[R_{i}]|_{\ast}\\
\leq & |\mathbb{E}_{i}[\mathcal{Y}_{i+1}]|_{\ast}+ \mathbb{E}_{i}%
[F_{i}]|_{\ast}+|\mathbb{E}_{i}[R_{i}]|_{\ast}.
\end{array}
\label{4-9}%
\end{equation}
Squaring the above (\ref{4-9}), then from the inequality $(\sum\limits_{i=1}%
^{n}a_{i})^{2}\leq n\sum\limits_{i=1}^{n}a_{i}^{2}$ and (\ref{4-8}), one
deduces
\begin{equation}%
\begin{array}
[c]{rl}%
|\mathbb{E}_{i}[\mathcal{Y}_{i}]|_{\ast}^{2}\leq & 3|\mathbb{E}_{i}%
[\mathcal{Y}_{i+1}]|_{\ast}^{2}+3|\mathbb{E}_{i}[F_{i}]|_{\ast}^{2}%
+3|\mathbb{E}_{i}[R_{i}]|_{\ast}^{2}\\
\leq & 3|\mathbb{E}_{i}[\mathcal{Y}_{i+1}]|_{\ast}^{2}+3c_{2} |\mathbb{E}%
_{i}[\gamma_{0}h\Delta\widehat{f}_{i}+\sum\limits_{j=1}^{m}\gamma_{j}h\Delta
f_{i+j}]|^{2}+3c_{2}|\mathbb{E}_{i}[\varepsilon_{i}^{Y}]|^{2}\\
\leq & 3|\mathbb{E}_{i}[\mathcal{Y}_{i+1}]|_{\ast}^{2}+3(m+1)h^{2}%
c_{2}\Big(|\mathbb{E}_{i}[\gamma_{0}\Delta\widehat{f}_{i}]|^{2}+\sum
\limits_{j=1}^{m}|\mathbb{E}_{i}[\gamma_{j}\Delta f_{i+j}]|^{2}\Big) +3c_{2}%
|\mathbb{E}_{i}[\varepsilon_{i}^{Y}]|^{2}.
\end{array}
\label{4-10}%
\end{equation}
By the Lipschitz condition of $f$ with respect to $(y,z)$ and
\[
\widetilde{Y}_{k}^{\pi}-\bar{\widetilde{Y}}_{k}^{\pi}=\mathbb{E}_{k}%
\Big[\sum_{j=1}^{m}\widetilde{\alpha}_{j}\Delta Y_{k+j}+\sum_{j=1}%
^{m}\widetilde{\gamma}_{j}h\Delta f_{k+j}\Big],
\]
(\ref{4-10}) can be restated as
\begin{align}
|\mathbb{E}_{i}[\mathcal{Y}_{i}]|_{\ast}^{2}\leq &  3|\mathbb{E}%
_{i}[\mathcal{Y}_{i+1}]|_{\ast}^{2}+6(m+1)h^{2}c_{2}L_{f}^{2} \Big(|\mathbb{E}%
_{i}[\gamma_{0}(\widetilde{Y}_{i}^{\pi}-\bar{\widetilde{Y}}_{i}^{\pi}%
)]|^{2}+\sum\limits_{j=1}^{m}|\mathbb{E}_{i}[\gamma_{j}\Delta Y_{i+j}%
]|^{2}\nonumber\\
&  +\sum\limits_{j=0}^{m}|\mathbb{E}_{i}[\gamma_{j}\Delta Z_{i+j}%
]|^{2}\Big)+3c_{2}|\mathbb{E}_{i}[\varepsilon_{i}^{Y}]|^{2}\nonumber\\
\leq &  3|\mathbb{E}_{i}[\mathcal{Y}_{i+1}]|_{\ast}^{2}+6(m+1)h^{2}c_{2}%
L_{f}^{2}\Big(\sum\limits_{j=1}^{m}(4m^{2}\gamma_{0}^{2}\widetilde{\alpha}%
_{j}^{2}+8m^{2}\gamma_{0}^{2}L_{f}^{2}h^{2}\widetilde{\gamma}_{j}^{2}%
+\gamma_{j}^{2})|\mathbb{E}_{i}[\Delta Y_{i+j}]|^{2}\nonumber\\
&  +\sum\limits_{j=0}^{m}(8m^{2}\gamma_{0}^{2}L_{f}^{2}h^{2}\widetilde{\gamma
}_{j}^{2}+\gamma_{j}^{2})|\mathbb{E}_{i}[\Delta Z_{i+j}]|^{2}\Big)+3c_{2}%
|\mathbb{E}_{i}[\varepsilon_{i}^{Y}]|^{2}\nonumber\\
\leq &  3|\mathbb{E}_{i}[\mathcal{Y}_{i+1}]|_{\ast}^{2}+6(m+1)h^{2}\frac
{c_{2}}{c_{1}}L_{f}^{2}\max\limits_{1\leq j\leq m}\{4m^{2}\gamma_{0}%
^{2}\widetilde{\alpha}_{j}^{2}+8m^{2}\gamma_{0}^{2}L_{f}^{2}h^{2}%
\widetilde{\gamma}_{j}^{2}+\gamma_{j}^{2}\}|\mathbb{E}_{i}[\mathcal{Y}%
_{i+1}]|_{\ast}^{2}\nonumber\\
&  +6(m+1)h^{2}c_{2}L_{f}^{2}\sum\limits_{j=0}^{m}(8m^{2}\gamma_{0}^{2}%
L_{f}^{2}h^{2}\widetilde{\gamma}_{j}^{2}+\gamma_{j}^{2})|\mathbb{E}_{i}[\Delta
Z_{i+j}]|^{2}+3c_{2} |\mathbb{E}_{i}[\varepsilon_{i}^{Y}]|^{2}. \label{4-11}%
\end{align}
\textbf{step 2.} Subtracting (\ref{4-4}) from (\ref{2-32}) with respect to
$Z$, we obtain
\begin{equation}
\Delta Z_{i}=\mathbb{E}_{i}\big[\sum_{j=1}^{m}\lambda_{m,j}\Delta
Y_{i+j}(W_{i+j}-W_{i})^{\top}\big]-\varepsilon_{i}^{Z}. \label{4-12}%
\end{equation}
Moreover, we get
\begin{equation}%
\begin{array}
[c]{rl}%
|\Delta Z_{i}|= & \Big|\mathbb{E}_{i}\big[\sum\limits_{j=1}^{m}\lambda
_{m,j}\Delta Y_{i+j}(W_{i+j}-W_{i})^{\top}\big]-\varepsilon_{i}^{Z}\Big|\\
\leq & \sum\limits_{j=1}^{m}\Big|\lambda_{m,j}\mathbb{E}_{i}\big[\Delta
Y_{i+j}(W_{i+j}-W_{i})^{\top}\big]\Big|+\big|\varepsilon_{i}^{Z}\big|.
\end{array}
\label{4-13}%
\end{equation}
Squaring the above equation (\ref{4-13}) and then by the Cauchy-Schwarz
inequality, we have
\begin{equation}%
\begin{array}
[c]{rl}%
|\Delta Z_{i}|^{2}\leq & (m+1)\sum\limits_{j=1}^{m}\max\limits_{1\leq j\leq
m}\{\lambda_{m,j}^{2}\}\Big|\mathbb{E}_{i}\big[\Delta Y_{i+j}(W_{i+j}%
-W_{i})^{\top}\big]\Big|^{2}+(m+1)\big|\varepsilon_{i}^{Z}\big|^{2}\\
\leq & (m+1)mdh\sum\limits_{j=1}^{m}\max\limits_{1\leq j\leq m}\{\lambda
_{m,j}^{2}\}\big|\mathbb{E}_{i}\big[\Delta Y_{i+j}\big]\big|^{2}%
+(m+1)\big|\varepsilon_{i}^{Z}\big|^{2}.
\end{array}
\label{4-14}%
\end{equation}
Summing over the above inequality from $i$ to $N-m$ and taking expectation, we
have
\begin{equation}%
\begin{array}
[c]{rl}%
\sum\limits_{k=i}^{N-m}h\mathbb{E}_{i}\big[|\Delta Z_{k}|^{2}\big]\leq &
(m+1)mdh^{2}\sum\limits_{k=i}^{N-m}\sum\limits_{j=1}^{m}\max\limits_{1\leq
j\leq m}\{\lambda_{m,j}^{2}\}\big|\mathbb{E}_{i}\big[\Delta Y_{i+j}%
\big]\big|^{2}+(m+1)h\sum\limits_{k=i}^{N-m}\mathbb{E}_{i}\big[|\varepsilon
_{k}^{Z}|^{2}\big]\\
\leq & (m+1)mdh^{2}\max\limits_{1\leq j\leq m}\{\lambda_{m,j}^{2}%
\}\sum\limits_{k=i+1}^{N-m+1}\mathbb{E}_{i}\big[|\mathcal{Y}_{k}%
|^{2}\big]+(m+1)h\sum\limits_{k=i}^{N-m}\mathbb{E}_{i}\big[|\varepsilon
_{k}^{Z}|^{2}\big].
\end{array}
\label{4-15}%
\end{equation}
\textbf{step 3.} Inserting (\ref{4-14}) into (\ref{4-11}), we obtain
\begin{align}
|\mathbb{E}_{i}[\mathcal{Y}_{i}]|_{\ast}^{2}\leq &  3|\mathbb{E}%
_{i}[\mathcal{Y}_{i+1}]|_{\ast}^{2}+6(m+1)h^{2}\frac{c_{2}}{c_{1}}L_{f}%
^{2}\max\limits_{1\leq j\leq m}\{4m^{2}\gamma_{0}^{2}\widetilde{\alpha}%
_{j}^{2}+8m^{2}\gamma_{0}^{2}L_{f}^{2}h^{2}\widetilde{\gamma}_{j}^{2}%
+\gamma_{j}^{2}\}|\mathbb{E}_{i}[\mathcal{Y}_{i+1}]|_{\ast}^{2}\nonumber\\
&  +6(m+1)^{2}mdh^{3}\frac{c_{2}}{c_{1}}L_{f}^{2}\max\limits_{1\leq j\leq
m}\{\lambda_{m,j}^{2}\}\max\limits_{1\leq j\leq m}\{8m^{2}\gamma_{0}^{2}%
L_{f}^{2}h^{2}\widetilde{\gamma}_{j}^{2}+\gamma_{j}^{2}\}\sum\limits_{j=0}%
^{m}\mathbb{E}_{i}\big[|\mathcal{Y}_{i+1+j}|_{\ast}^{2}\big]\nonumber\\
&  +6(m+1)^{2}h^{2}c_{2}L_{f}^{2}\sum\limits_{j=0} ^{m}(8m^{2}\gamma_{0}%
^{2}L_{f}^{2}h^{2}\widetilde{\gamma}_{j}^{2}+\gamma_{j}^{2})\mathbb{E}%
_{i}\big[|\varepsilon_{i+j}^{Z}|^{2}\big]+3c_{2}|\mathbb{E}_{i}[\varepsilon
_{i}^{Y}]|^{2}\nonumber\\
\leq &  3|\mathbb{E}_{i}[\mathcal{Y}_{i+1}]|_{\ast}^{2}+6(m+1)h^{2}\frac
{c_{2}}{c_{1}}L_{f}^{2}\Big(\max\limits_{1\leq j\leq m}\{4m^{2}\gamma_{0}%
^{2}\widetilde{\alpha}_{j}^{2}+8m^{2}\gamma_{0}^{2}L_{f}^{2}h^{2}%
\widetilde{\gamma}_{j}^{2}+\gamma_{j}^{2}\}\nonumber\\
&  +(m+1)^{2}mdh\max\limits_{1\leq j\leq m}\{\lambda_{m,j}^{2}\}\max
\limits_{1\leq j\leq m}\{8m^{2}\gamma_{0}^{2}L_{f}^{2}h^{2}\widetilde{\gamma
}_{j}^{2}+\gamma_{j}^{2}\}\Big)\sum_{k=i+1}^{N-m+1}|\mathbb{E}_{i}%
[\mathcal{Y}_{k}]|_{\ast}^{2}\nonumber\\
&  +6(m+1)^{2}h^{2}c_{2}L_{f}^{2}\max\limits_{1\leq j\leq m}\{8m^{2}\gamma
_{0}^{2}L_{f}^{2}h^{2}\widetilde{\gamma}_{j}^{2}+\gamma_{j}^{2}\}
\sum\limits_{j=0}^{m}\mathbb{E}_{i}\big[|\varepsilon_{i+j}^{Z}|^{2}%
\big]+3c_{2}|\mathbb{E}_{i}[\varepsilon_{i}^{Y}]|^{2}. \label{4-16}%
\end{align}
There exists a constant $C$ which changes from line to line such that
\begin{equation}
|\mathbb{E}_{i}[\mathcal{Y}_{i}]|_{\ast}^{2}\leq C\Big((h+h^{2})\sum
_{k=i+1}^{N-m+1}|\mathbb{E}_{i}[\mathcal{Y}_{k}]|_{\ast}^{2}+\sum_{k=0}%
^{m}\mathbb{E}_{i}\big[|\varepsilon_{k}^{Y}|^{2}+h^{2}|\varepsilon_{i+k}%
^{Z}|^{2}\big]\Big). \label{4-17}%
\end{equation}
From Lemma \ref{est-1}, we have
\begin{equation}
|\mathbb{E}_{i}[\mathcal{Y}_{i}]|_{\ast}^{2}\leq C\Big( \max\limits_{i+1\le k
\le N-m+1}mh|\mathbb{E}_{i}[\mathcal{Y}_{k}]|_{\ast}^{2}+ \sum_{k=0}%
^{m}\mathbb{E}_{i}\big[|\varepsilon_{k}^{Y}|^{2}+h^{2}|\varepsilon_{i+k}%
^{Z}|^{2}\big]\Big). \label{4-18}%
\end{equation}
Inserting (\ref{4-18}) into (\ref{4-15}), we get, for $h$ small enough
\begin{equation}%
\begin{array}
[c]{l}%
\sum\limits_{k=i}^{N-m}h\mathbb{E}_{i}\big[ |\Delta Z_{k}|^{2}\big] \leq
C\left( \max\limits_{i+1\le k \le N}|\mathbb{E}[Y_{k}]|^{2} +\sum
\limits_{k=i}^{N-m}\mathbb{E}_{i}\big[h|\varepsilon_{k}^{Z} |^{2}+\frac{1}%
{h}|\varepsilon_{k}^{Y}|^{2}\big]\right) .
\end{array}
\label{4-19}%
\end{equation}
Adding (\ref{4-18}) to the above (\ref{4-19}), we derive that there exists a
constant $C$ such that
\[
\max_{0\leq i\leq N-m}|\mathbb{E}[\mathcal{Y}_{i}]|_{\ast}^{2}+\sum
_{i=0}^{N-m}h\mathbb{E}\big[|\Delta Z_{i}|^{2}\big]\leq C \left(
\max\limits_{N-m+1\le k \le N}|Y_{k}^{\pi}-\bar{Y}_{k}^{\pi}|^{2} +\sum
_{i=0}^{N-m}\mathbb{E}_{i}\big[h|\varepsilon_{i}^{Z}|^{2}+\frac{1}%
{h}|\varepsilon_{i}^{Y}|^{2}\big]\right) .
\]
\textbf{Necessity:} The proof is analogous to ordinary differential equations
(see Theorem 6.3.3 of \cite{16}). So we omit it.
\end{proof}

\begin{theorem}
Suppose that Assumption 2 holds. Furthermore, $f(t,x,y,z)$ and $\Phi(x_{T})$
are smooth enough functions. Let $(Y_{t_{i}},Z_{t_{i}})$ and $(Y_{i}^{\pi
},Z_{i}^{\pi})$ be solutions of the BSDE in (\ref{2-1}) and solutions of the
scheme (\ref{2-32}) respectively. The terminal values satisfy $\mathbb{E}%
[\sup\limits_{N-m< i\le N}|Y_{i}^{\pi}-Y_{t_{i}}|^{2} + h|Z_{i}^{\pi}
-Z_{t_{i}}|^{2}]^{\frac{1}{2}}\le Ch^{m+1}$. Then, as $h$ is small enough
\begin{equation}
\mathbb{E}[\sup_{0\le i\le N-m}|Y_{i}^{\pi}-Y_{t_{i}}|^{2} + h|Z_{i}^{\pi
}-Z_{t_{i}}|^{2}]^{\frac{1}{2}}\le Ch^{m+1},\nonumber
\end{equation}
where $C$ is a constant changing from line to line.
\end{theorem}

\begin{proof}
The BSDE in (\ref{2-1}) is discretized by the scheme as below:
\begin{equation}
\left\{
\begin{array}
[c]{rl}%
\widetilde{Y}_{i}= & \mathbb{E}_{i}\Big[\sum\limits_{j=1}^{m}
\widetilde{\alpha}_{j}Y_{t_{i+j}}+ \sum\limits_{j=1}^{m} \widetilde{\gamma
}_{j}hf_{i+j}\Big],\\
Y_{t_{i}}= & \mathbb{E}_{i}\Big[\sum\limits_{j=1}^{m} \alpha_{j}Y_{t_{i+j}}+
\gamma_{0}h\widetilde{f}_{i}+ \sum\limits_{j=1}^{m} \gamma_{j}hf_{i+j}%
\Big]+R_{Y,i},\\
Z_{t_{i}}= & \mathbb{E}_{i}\big[\sum\limits_{j=1}^{m}\lambda_{m,j}Y_{t_{i+j}%
}(W_{i+j}-W_{i})^{\top}\big]+R_{Z,i},
\end{array}
\right.  \label{3-7}%
\end{equation}
where $R_{Y,i}$ and $R_{Z,i}$ denote the error of the exact solutions and the
approximation solutions w.r.t. $Y$ and $Z$; $f_{i}=f(t_{i},X_{t_{i}},Y_{t_{i}%
},Z_{t_{i}})$, $\widetilde{f}_{i}=f(t_{i},X_{t_{i}},\widetilde{Y}_{t_{i}%
},Z_{t_{i}})$, $i=0,1,\cdots,N$. Set $\Delta\mathcal{Y}_{i}=Y_{i}^{\pi
}-Y_{t_{i}},\Delta\widetilde{\mathcal{Y}}_{i}=\widetilde{Y}_{i}^{\pi
}-\widetilde{Y}_{t_{i}},\Delta\mathcal{Z}_{i}=Z_{i}^{\pi}-Z_{t_{i}}$. From
(\ref{2-32}) and (\ref{3-7}), we have
\begin{equation}%
\begin{array}
[c]{rl}%
|\Delta\mathcal{Y}_{i}|= & \Big|\mathbb{E}_{i}\big[\sum\limits_{j=1}^{m}%
\alpha_{j}\Delta\mathcal{Y}_{i+j}+h \gamma_{0}(\widetilde{f}_{i}^{\pi
}-\widetilde{f}_{i})+h \sum\limits_{j=1}^{m}\gamma_{j}(f_{i+j}^{\pi}%
-f_{i+j})\big]-R_{Y,i}\Big|\\
\leq & \mathbb{E}_{i}\big[\sum\limits_{j=1}^{m}|\alpha_{j}||\Delta
\mathcal{Y}_{i+j}|+h|\gamma_{0}||\widetilde{f}_{i}^{\pi}-\widetilde{f}%
_{i}|+h\sum\limits_{j=1}^{m}|\gamma_{j}||f_{i+j}^{\pi}-f_{i+j}|\big]+|R_{Y,i}%
|\\
\leq & \mathbb{E}_{i}\big[\sum\limits_{j=1}^{m}|\alpha_{j}||\Delta
\mathcal{Y}_{i+j}|+h|\gamma_{0} L_{f}||\Delta\widetilde{\mathcal{Y}}%
_{i}+\Delta\mathcal{Z}_{i}| +h\sum\limits_{j=1}^{m}|\gamma_{j}L_{f}%
||\Delta\mathcal{Y}_{i+j}+\Delta\mathcal{Z}_{i+j}|\big]+|R_{Y,i}|.
\end{array}
\label{3-10}%
\end{equation}
\begin{equation}%
\begin{array}
[c]{rl}%
|\Delta\widetilde{\mathcal{Y}}_{i}| = |\widetilde{Y}_{i}^{\pi}- \widetilde{Y}%
_{t_{i}}| = & \big|\mathbb{E}_{i}[\sum\limits_{j=1}^{m}\widetilde{\alpha}%
_{j}\Delta\mathcal{Y}_{i+j} +h\sum\limits_{j=1}^{m}\widetilde{\gamma}%
_{j}(f_{i+j}^{\pi}- f_{i+j})]\big|\\
\le & \mathbb{E}_{i}[\sum\limits_{j=1}^{m}|\widetilde{\alpha}_{j}%
||\Delta\mathcal{Y}_{i+j}| +h\sum\limits_{j=1}^{m}|\widetilde{\gamma}_{j}%
L_{f}||\Delta\mathcal{Y}_{i+j} + \Delta\mathcal{Z}_{i+j}|].
\end{array}
\label{3-11}%
\end{equation}

\begin{equation}
\Delta\mathcal{Z}_{i} = \sum\limits_{j=1}^{m}\lambda_{m,j}\mathbb{E}%
_{i}\big[\Delta\mathcal{Y}_{i+j} (W_{i+j}-W_{i})^{\top}\big]-R_{Z,i}.
\label{3-12}%
\end{equation}
Inserting (\ref{3-11}) into (\ref{3-10}), we obtain
\begin{equation}%
\begin{array}
[c]{rl}%
|\Delta\mathcal{Y}_{i}| \le & \mathbb{E}_{i}\big[\sum\limits_{j=1}^{m}%
|\alpha_{j}||\Delta\mathcal{Y}_{i+j}| + h |\gamma_{0}L_{f}||\mathbb{E}%
_{i}[\sum\limits_{j=1}^{m}|\widetilde{\alpha}_{j}||\Delta\mathcal{Y}_{i+j}|
+h\sum\limits_{j=1}^{m}|\widetilde{\gamma}_{j}L_{f}||\Delta\mathcal{Y}_{i+j} +
\Delta\mathcal{Z}_{i+j}|]\\
& + \Delta\mathcal{Z}_{i}|+ h\sum\limits_{j=1}^{m}|\gamma_{j}L_{f}%
||\Delta\mathcal{Y}_{i+j} + \Delta\mathcal{Z}_{i+j}|\big]+ | R_{Y,i}|\\
\le & \mathbb{E}_{i}\big[\sum\limits_{j=1}^{m}(|\alpha_{j}|+hL_{f} |\gamma
_{0}\widetilde{\alpha}_{j}| )|\Delta\mathcal{Y}_{i+j}| + h |\gamma_{0}%
L_{f}||\Delta\mathcal{Z}_{i}|\\
& +\sum\limits_{j=1}^{m}\left( h^{2}L_{f}^{2}|\gamma_{0}\widetilde{\gamma}%
_{j}| +hL_{f}|\gamma_{j}|\right) |\Delta\mathcal{Y}_{i+j} + \Delta
\mathcal{Z}_{i+j}|\Big]+ |R_{Y,i}|.
\end{array}
\label{3-13}%
\end{equation}
Squaring the inequality (\ref{3-13}) and inserting (\ref{3-12}) into the
derived equation yield, for $i = N-m,N-m-1,\cdots,0$
\begin{align}
|\Delta\mathcal{Y}_{i}|^{2} \leq &  4\bigg(\mathbb{E}_{i}\Big[\sum
\limits_{j=1}^{m}m(|\alpha_{j}|+hL_{f} |\gamma_{0}\widetilde{\alpha}_{j}|
)^{2}|\Delta\mathcal{Y}_{i+j}|^{2} + h^{2}L_{f}^{2}\gamma_{0}^{2}
|\sum\limits_{j=1}^{m}\lambda_{m,j}\mathbb{E}_{i}\big[\Delta\mathcal{Y}_{i+j}
(W_{i+j}-W_{i})^{\top}\big]-R_{Z,i}|^{2}\nonumber\\
& +\sum\limits_{j=1}^{m}2m\left( h^{2}L_{f}^{2}|\gamma_{0}\widetilde{\gamma
}_{j}| +hL_{f}|\gamma_{j}|\right) ^{2}\Big(|\Delta\mathcal{Y}_{i+j}%
|^{2}\nonumber\\
&  +|\sum\limits_{n=1}^{m}\lambda_{m,n}\mathbb{E}_{i+j}\big[\Delta
\mathcal{Y}_{i+j+n} (W_{i+j+n}-W_{i+j})^{\top}\big]-R_{Z,i+j}|^{2}%
\Big) +|R_{Y,i}|^{2}\bigg)\nonumber\\
\leq &  4\bigg(\mathbb{E}_{i}\Big[\Big(m(|\alpha_{j}|+hL_{f} |\gamma
_{0}\widetilde{\alpha}_{j}| )^{2} + hm(m+1)L_{f}^{2}d\gamma_{0}^{2}%
(\max\limits_{1\le j\le m}\lambda_{m,j}h)^{2} + 4mh^{2}L_{f}^{2}(h^{2}%
L_{f}^{2}|\gamma_{0}\widetilde{\gamma}_{j}|^{2}+|\gamma_{j}|^{2})\nonumber\\
& +4m^{2}(m+1)dhL_{f}^{2}(\max\limits_{1\le j\le m}\lambda_{m,j}h)^{2}%
(h^{2}L_{f}^{2}|\gamma_{0}\widetilde{\gamma}_{j}|^{2}+|\gamma_{j}%
|^{2})\Big) \sum\limits_{j=1}^{m}|\Delta\mathcal{Y}_{i+j}|^{2} \Big]+R_{Y,i}%
^{2}\nonumber\\
& +\Big(h^{2}(m+1)L_{f}^{2}\gamma_{0}^{2} + 4m(m+1)h^{2}L_{f}^{2}(h^{2}%
L_{f}^{2}|\gamma_{0}\widetilde{\gamma}_{j}|^{2}+|\gamma_{j}|^{2}%
)\Big) \max\limits_{0\le j\le m}R_{Z,i+j}^{2}\bigg)\nonumber\\
\leq &  C\bigg(\mathbb{E}_{i}\Big[(h+h^{2})\sum\limits_{j=i+1}^{N}%
|\Delta\mathcal{Y}_{j}|^{2}\Big] +|R_{Y,i}|^{2}+h^{2}\max\limits_{i+1\le j\le
N}R_{Z,j}^{2} +O(h^{2m+2})\bigg). \label{3-14}%
\end{align}
From Lemma \ref{est-1}, the inequality (\ref{3-14}) can be rewritten as
\begin{equation}%
\begin{array}
[c]{rl}%
|\Delta\mathcal{Y}_{i}|^{2}\leq & C\bigg(mh\max\limits_{N-m\le j \le N}%
|\Delta\mathcal{Y}_{j}|^{2} +|R_{Y,i}|^{2}+h^{2}\max\limits_{i+1\le j\le
N}R_{Z,j}^{2} +O(h^{2m+2})\bigg).
\end{array}
\label{3-14-1}%
\end{equation}
From Proposition \ref{pro-pc}, we derive
\begin{equation}
|R_{Y,i}|\le Ch^{m+1}.\label{errrorofY}%
\end{equation}
By Lemma 2.1 in \cite{40}, we have
\begin{equation}
|R_{Z,i}|\le Ch^{m}.\label{errrorofZ}%
\end{equation}
Combining (\ref{3-14-1}), (\ref{errrorofY}) with (\ref{errrorofZ}), we deduce
that
\begin{equation}
|\Delta\mathcal{Y}_{i}|^{2}\leq C(h^{2m+3} + h^{2m+2}),\label{3-15}%
\end{equation}
for $i = N-m-1,N-m-2,\cdots,0$ recursively. Hence, $\sup\limits_{0\leq i<
N-m}|\Delta\mathcal{Y}_{i}|\leq Ch^{m+1}$.

Squaring (\ref{3-12}) and multiplying $h$ and then with the help of
Cauchy-Schwarz inequality, we have
\begin{align}
h|\Delta\mathcal{Z}_{i}|^{2}=  & h|\sum\limits_{j=1}^{m}\lambda_{m,j}%
\mathbb{E}_{i}\big[\Delta\mathcal{Y}_{i+j} (W_{i+j}-W_{i})^{\top}%
\big]-R_{Z,i}|^{2}\nonumber\\
\le & h(m+1)\left( \sum\limits_{j=1}^{m}\lambda_{m,j}^{2}mhd\mathbb{E}%
_{i}[|\Delta\mathcal{Y}_{i+j}|^{2}]+|R_{Z,i}|^{2}\right) \nonumber\\
\le & C\left( \max\limits_{i+1\le j\le N}\mathbb{E}_{i}[|\Delta\mathcal{Y}%
_{j}|^{2}]+h|R_{Z,i}|^{2}\right) \nonumber\\
\leq & Ch^{2m+2}.\label{3-16}%
\end{align}
Hence, we deduce the conclusion with the help of (\ref{3-15}) and
(\ref{3-16}). The proof is completed.
\end{proof}

\section{Numerical Experiments}

In this section, we provide two numerical examples to show the performance of
the scheme (\ref{2-32}). Specifically, in the \textbf{Example 1}, we provide
stable numerical schemes for the step number $m=1,2,3,4$ to show their
convergence rates w.r.t. the time step sizes, absolute errors and running
times. And the comparisons with explicit Adams method in \cite{9} are also
given. In the \textbf{Example 2}, we also present unstable numerical schemes
for the step number $m=2,3$ to illustrate the previous theory analysis.

To assess the performance of our algorithms, we had better to find a BSDE with
closed-form solutions and establish criterions. Let $\epsilon=\mathbb{E}%
[|Y_{0}-Y_{0}^{\pi}|]$ denote the error between closed-form solutions and
numerical solutions. From the Central Limit Theorem, one gets the error
$\bar{\epsilon}:=\frac{1}{M}\sum\limits_{k=1}^{M}|Y_{0}-Y_{0,k}^{\pi}|$ that
converges in distribution to $\epsilon$ as $M\rightarrow\infty$.

In implementation, one can calculate the variance $\hat{\sigma}_{\epsilon}%
^{2}$ of $\hat{\epsilon}$ and then utilize it to construct a confidence
interval (CI) for the absolute error $\epsilon$. To realize this idea, one
arranges the simulations into $\widetilde{M}$ batches of $M$ simulations each
and estimates the variance $\hat{\sigma}_{\epsilon}^{2}$. To be precise,
define the average errors $\hat{\epsilon}_{j} = \frac{1}{M}\sum\limits_{k=1}%
^{M}|Y_{0}- Y_{0,k,j}^{\pi}|,j =1,2,\cdots,\widetilde{M}$, where
$Y_{0,k,j}^{\pi}$ is $k$-th trajectory generated by our schemes in the $j$th
batch at time $0$. These average errors are independent and approximately
Gaussian when $M$ is large enough. Thus, the mean of the batch averages is
$\hat{\epsilon}=\frac{1}{\widetilde{M}}\sum\limits_{j=1}^{\widetilde{M}}%
\hat{\epsilon}_{j} = \frac{1}{M\widetilde{M}}\sum\limits_{j=1}^{\widetilde{M}%
}\sum\limits_{k=1}^{M}|Y_{0}- Y_{0}^{\pi}|$ and the variance of the batch
averages is $\hat{\sigma}_{\epsilon}^{2} = \frac{1}{\widetilde{M}-1}%
\sum\limits_{j=1}^{\widetilde{M}} (\hat{\epsilon}_{j}-\hat{\epsilon} )^{2}$.
Experience has shown that the batch averages can be interpreted as being
Gaussian for batch sizes $\widetilde{M}\ge15$. A $1-\alpha$ confidence
interval for $\epsilon$ has the form $(\hat{\epsilon}-t_{1-\alpha
,\widetilde{M}-1}\sqrt{\frac{\hat{\sigma}_{\epsilon}^{2}}{\widetilde{M}}},
\hat{\epsilon}+ t_{1-\alpha,\widetilde{M}-1}\sqrt{\frac{\hat{\sigma}%
_{\epsilon}^{2}}{\widetilde{M}}})$ where $t_{1-\alpha,\widetilde{M}-1}$ is
determined from $t$-distribution with $\widetilde{M}-1$ degrees of freedom.

Next, algorithms are founded via our schemes, and the emerged conditional
expectations in our schemes are simulated by means of least squares Monte
Carlo method (see \cite{4,18,20,21}). Let $OLS$ denote the ordinary least
squares. Define the empirical probability measure $\nu_{i,M}=\frac{1}{M}%
\sum\limits_{\hat{m}=1}^{M}\delta_{(\Delta W_{i}^{(i,\hat{m})},X_{i}%
^{(i,\hat{m})},\cdots,X_{N}^{(i,\hat{m})})}$ where $\delta_{x}$ is the Dirac
measure and $\{(\Delta W_{i}^{(i,\hat{m})},X^{(i,\hat{m})}):\hat{m}%
=1,2,\cdots,M\}$ is the independent copies of $(\Delta W_{i},X)$; the finite
functional linear space $\mathcal{K}_{Y,i}:=\{p_{Y,i}^{(1)}(\cdot
),p_{Y,i}^{(2)}(\cdot),\cdots,p_{Y,i}^{(K_{Y,i})}(\cdot)\}$, the basis
function $p_{Y,i}^{(k)}:\mathbb{R}^{d}\rightarrow\mathbb{R}$ such that
$\mathbb{E}[|p_{Y,i}^{(k)}(X_{i})|^{2}]<+\infty$ and the finite functional
linear space $\mathcal{K}_{Z,i}:=\{p_{Z,i}^{(1)}(\cdot),p_{Z,i}^{(2)}%
(\cdot),\cdots,p_{Z,i}^{(K_{Z,i})}(\cdot)\}$, the basis function
$p_{Z,i}^{(k)}:\mathbb{R}^{d}\rightarrow\mathbb{R}^{d}$ such that
$\mathbb{E}[|p_{Z,i}^{(k)}(X_{i})|^{2}]<+\infty$ where $K_{Y,i}$ and $K_{Z,i}$
denote the dimension of the finite functional linear spaces $\mathcal{K}%
_{Y,i}$ and $\mathcal{K}_{Z,i}$. Suppose that $\mathcal{T}_{L}(x)$ is the
truncation operator and it is defined as $\mathcal{T}_{L}(x)= (-L \vee x_{1}
\wedge L,\cdots,-L \vee x_{n} \wedge L) $ for any finite $L>0$, $x=
(x_{1},\cdots,x_{n})\in\mathbb{R}^{n}$. Note that there are measurable,
deterministic (but unknown) functions $y_{i}(\cdot):\mathbb{R}^{d}%
\rightarrow\mathbb{R}$ and $z_{i}(\cdot):\mathbb{R}^{d}\rightarrow
\mathbb{R}^{d}$ for $i=0,1,\cdots,N-1$ such that the solution $(Y_{i}^{\pi
},Z_{i}^{\pi})$ of the discrete BSDE (\ref{2-32}) is given by $(Y_{i}^{\pi
},Z_{i}^{\pi}): = (y_{i}(X_{i}^{\pi}),z_{i}(X_{i}^{\pi}))$ (see Theorem 3.1 in
\cite{CBTM10}).

\begin{table}[h]
\centering\resizebox{\textwidth}{35mm}{
\begin{tabular}
[c]{ccccccccc}\hline
\textbf{Algorithm } the stable high order predictor-corrector scheme based on
(\ref{2-32}) &  &  &  &  &  &  &  & \\\hline
\multicolumn{9}{l}{1. Initialization $X_{0}:=x_{0},y_{N}^{(M)}(\cdot):=\Phi(X_{N}^{\pi})$}\\
\multicolumn{9}{l}{2. sample $(t_{i},x_{i}^{\widehat{m}})_{1\leq i\leq N}$ by
$X_{i+1}^{\pi,\widehat{m}}=X_{i}^{\pi,\widehat{m}}+b(t_{i},X_{i}^{\pi,\widehat{m}})h+\sigma(t_{i},X_{i}^{\pi,\widehat{m}})\Delta W_{i+1}$}\\
\multicolumn{9}{l}{3. for $i=N-1$ until $1$}\\
\multicolumn{9}{l}{4. for $\widehat{m}=1$ until $M$}\\
\multicolumn{9}{l}{5. set $S_{Z,i}^{(M)}(\mathbf{W},\mathbf{X})=\sum\limits
_{j=1}^{m}\lambda_{m,j}y_{i+j}^{(M)}(X_{i+j}^{\pi,\widehat{m}})(W_{i+j}-W_{i})^{\top}$,
where $\mathbf{W} = (W_i,\cdots,W_{i+m+1})\in (\mathbb{R}^d)^{m+2},$
$\mathbf{X} = (X_i^\pi,\cdots,X^\pi_{i+m+1})\in (\mathbb{R}^d)^{m+2}$;}\\
\multicolumn{9}{l}{\quad  compute $z_{i}^{(M)}(X_{i}^{\pi})
=\mathcal{T}_{C_z}\Big(OLS\big(S_{Z,i}^{(M)},\mathcal{K}_{Z,i},\nu
_{i,M}\big)\Big)$, where $C_z$ denotes  the upper bound of Z}\\
\multicolumn{9}{l}{6. set $\widetilde{S}_{Y,i}^{(M)}(\mathbf{X})=\sum\limits
_{j=1}^{m}y_{i+j}^{(M)}(X_{i+j}^{\pi,\widehat{m}})+h\sum\limits_{j=1}^{m}\widetilde{\gamma}_{j}f(t_{i+j}, X_{i+j}^{\pi,\widehat{m}}, y_{i+j}^{(M)}(X_{i+j}^{\pi,\widehat{m}}),z_{i+j}^{(M)}(X_{i+j}^{\pi,\widehat{m}}))$, compute $\widetilde{y}_{i}^{(M)}(X_{i}^{\pi})=\mathcal{T}_{C_y}\Big(OLS\big(\widetilde{S}_{Y,i}^{(M)},\mathcal{K}_{Y,i},\nu_{i,M}\big)\Big)$}\\
\multicolumn{9}{l}{\quad  where $C_y$ denotes  the upper bound of Y}\\
\multicolumn{9}{l}{7. set $S_{Y,i}^{(M)}(\mathbf{X})=\sum\limits_{j=1}^{m}\alpha
_{j}y_{i+j}^{(M)}(X_{i+j}^{\pi,\widehat{m}})+h\gamma_{0}f(t_{i}, X_{i}^{\pi,\widehat{m}}, \widetilde{y}_{i}^{(M)}(X_{i}^{\pi,\widehat{m}}),z_{i}^{(M)}(X_{i}^{\pi,\widehat{m}}))
+h\sum\limits_{j=1}^{m}\gamma_{j}f(t_{i+j}, X_{i+j}^{\pi,\widehat{m}}, y_{i+j}^{(M)}(X_{i+j}^{\pi,\widehat{m}}),z_{i+j}^{(M)}(X_{i+j}^{\pi,\widehat{m}}))$,}\\
\multicolumn{9}{l}{\quad compute $y_{i}^{(M)}(X_{i}^{\pi})=\mathcal{T}
_{C_y}\Big(OLS\big(S_{Y,i}^{(M)},\mathcal{K}_{Y,i},\nu_{i,M}\big)\Big)$}\\
\multicolumn{9}{l}{8. end for }\\
\multicolumn{9}{l}{9. end for }\\\hline
\end{tabular}}\end{table}

In what follows, we apply our $\mathbf{Algorithms}$ to two BSDEs with
closed-form solutions.

\textbf{Example 1.} Consider the BSDE as below:
\begin{equation}
Y_{t} = 1+\eta+sin(\tau\mathbf{1}_{d}^{\top}W_{T}) +\int_{t}^{T}%
\min\big\{1,(Y_{s}-\eta-1-\frac{sin(\tau\mathbf{1}_{d}^{\top}W_{s})}{\exp
(\tau^{2}d(T-t)/2)})^{2}\big\}ds - \int_{t}^{T}Z_{s}dW_{s},\label{BSDEdetail}%
\end{equation}
which appears in \cite{EGPT17} and is used to illustrate the variance
reduction problem with closed-form solutions. Here $\eta>0$; $\tau>0$;
$\mathbf{1}_{d}$ is a $d$-dimensional vector with components all 1. Now, the
solution to the above BSDE is
\[
Y_{t}=1+\eta+\frac{sin(\tau\mathbf{1}_{d}^{\top}W_{t})}{\exp(\tau
^{2}d(T-t)/2)},~~~~~~(Z_{t})_{\lambda}=\frac{\tau cos(\tau\mathbf{1}_{d}%
^{\top}W_{t})}{\exp(\tau^{2}d(T-t)/2)},
\]
where $(Z_{t})_{\lambda}$ is the $\lambda$th component of the $d$-dimensional
function $Z_{t}\in\mathbb{R}^{d}$. Take $T=1,\eta=0.6,\tau=\frac{1}{\sqrt{d}},
d=2,\widetilde{M}=21,h= \frac{T}{N}$. The basis functions which are spanned by
polynomials whose degree is $2$ are applied to compute the value of
$Y_{i}^{(M)}$ and $Z_{i}^{(M)}$. In the Tables, the notations CR and RT
represent the convergence rate w.r.t. the time step sizes and the running time
respectively. The unit of RT is the second. In the Figures, the notations GPC
scheme and EAM scheme represent the scheme (\ref{2-32}) and the usual explicit
Adams methods from \cite{9} respectively.

If we want to implement the \textbf{Algorithm }, we have to determine
parameters. Specifically, the following equations should be satisfied for
$m=1$%
\[
\left\{
\begin{array}
[c]{l}%
0=1-\widetilde{\alpha}_{1},\\
0=-\widetilde{\alpha}_{1}+\widetilde{\gamma}_{1},\\
|\widetilde{\alpha}_{1}|\leq1,
\end{array}
\right.  \left\{
\begin{array}
[c]{l}%
0=1-\alpha_{1},\\
0=-\alpha_{1}+\gamma_{0}+\gamma_{1},\\
|\alpha_{1}|\leq1,
\end{array}
\right.  and\left\{
\begin{array}
[c]{l}%
0=\lambda_{1,0}h+\lambda_{1,1}h,\\
1=\lambda_{1,1}h.
\end{array}
\right.
\]
Thus, $\widetilde{\alpha}_{1}=1,\widetilde{\gamma}_{1}=1,\alpha_{1}=1.$ Let
$\gamma_{0}=\frac{1}{2}$, then $\gamma_{1}=\frac{1}{2}.$ $\lambda
_{1,0}h=-1,\lambda_{1,1}h=1.$ Now, the characteristic polynomial becomes
$P(\zeta) = \zeta- 1$. Its root $1$ fulfils Dahlquist's root condition. That
is to say, this one-step scheme is stable and given as
\begin{equation}
\left\{
\begin{array}
[c]{rl}%
\widetilde{Y}_{i}^{\pi}= & \mathbb{E}_{i}\Big[Y_{i+1}^{\pi}+ hf_{i+1}^{\pi
}\Big],\\
Y_{i}^{\pi}= & \mathbb{E}_{i}\Big[Y_{i+1}^{\pi}+\frac{1}{2}h\widetilde{f}%
_{i}^{\pi} +\frac{1}{2}hf_{i+1}^{\pi}\Big],\\
Z_{i}^{\pi}= & \mathbb{E}_{i}\big[Y_{i+1}^{\pi}\frac{(W_{i+1}-W_{i})^{\top}%
}{h}\big].
\end{array}
\right. \nonumber
\end{equation}
Analogously, we present the stable predictor-corrector type general linear
multi-step scheme for $m=2,3,4$. For example, if $m=3$, we provide the
following three-step scheme
\begin{equation}
\left\{
\begin{array}
[c]{rl}%
\widetilde{Y}_{i}^{\pi}= & \mathbb{E}_{i}\Big[\frac{1}{3}Y_{i+1}^{\pi}%
+\frac{1}{3}Y_{i+2}^{\pi}+\frac{1}{3}Y_{i+3}^{\pi} + \frac{39}{18}%
hf_{i+1}^{\pi}-\frac{2}{3}hf_{i+2}^{\pi}+\frac{1}{2}hf_{i+3}^{\pi}\Big],\\
Y_{i}^{\pi}= & \mathbb{E}_{i}\Big[\frac{1}{3}Y_{i+1}^{\pi}+\frac{1}{3}%
Y_{i+2}^{\pi}+\frac{1}{3}Y_{i+3}^{\pi} +\frac{5}{6}h\widetilde{f}_{i}^{\pi
}-\frac{1}{3}hf_{i+1}^{\pi}+\frac{11}{6}hf_{i+2}^{\pi}-\frac{1}{3}%
hf_{i+3}^{\pi}\Big],\\
Z_{i}^{\pi}= & \mathbb{E}_{i}\big[3Y_{i+1}^{\pi}\frac{(W_{i+1}-W_{i})^{\top}%
}{h} -\frac{3}{2}Y_{i+2}^{\pi}\frac{(W_{i+2}-W_{i})^{\top}}{h} + \frac{1}%
{3}Y_{i+3}^{\pi}\frac{(W_{i+3}-W_{i})^{\top}}{h}\big].
\end{array}
\right. \nonumber
\end{equation}
Now, the characteristic polynomial becomes $P(\zeta) = \zeta^{3}- \frac{1}%
{3}\zeta^{2}- \frac{1}{3}\zeta- \frac{1}{3}$. Its roots $1,- \frac{1}{3} +
\frac{1121}{2378}i, - \frac{1}{3} - \frac{1121}{2378}i$ fulfil Dahlquist's
root condition. That is to say, this three-step scheme is stable.

\begin{table}[ptb]
\caption{ Errors and convergence rates based on the \textbf{ Algorithm }}%
\centering\resizebox{\textwidth}{66mm}{
\begin{tabular}
[c]{c|cccccc|c}\hline
Step & N & M &$|Y_0-Y_{0}^{(M)}|$ & 95\%CI of Y & $|Z_0-Z_{0}^{(M)}|$ & 95\%CI of Z&RT\\\hline
\multirow{5}{*}{1} & 5 & 2778 & 1.257e-02 & (8.461e-03, 1.668e-02) & 1.279e-02 & (8.685e-03,
1.690e-02)&0.1275\\
& 10 & 5996 & 7.969e-03 & (5.449e-03, 1.049e-02) & 8.621e-03 & (5.526e-03,
1.172e-02)&0.4401\\
& 15 & 8809 & 6.877e-03 & (4.757e-03, 8.998e-03) & 7.393e-03 & (4.432e-03,
1.035e-02)&1.548\\
& 20 & 12018 & 5.276e-03 & (3.966e-03, 6.586e-03) & 6.158e-03 & (3.734e-03,
8.582e-03)&2.992\\\cline{2-8}
&  \multicolumn{2}{c}{CR}  &1.021 &  & 1.004& &\\\hline
\multirow{5}{*}{2} & 5 & 2778 & 7.960e-03 & (5.440e-03, 1.048e-02) & 8.776e-03 & (6.257e-03,
1.130e-02)&0.1779\\
& 10 & 5996 & 7.012e-04 & (4.617e-04, 1.081e-03) & 7.193e-03 & (4.672e-03,
9.713e-03)&0.9788\\
& 15 & 8809 & 6.128e-04 & (4.006e-04, 8.250e-04) & 6.328e-04 & (3.887e-04,
8.769e-04)&2.254\\
& 20 & 12018 & 4.276e-04 & (2.967e-04, 5.586e-04) & 4.450e-04 & (2.150e-04
6.749e-04)&3.795\\\cline{2-8}
&  \multicolumn{2}{c}{CR}  &1.998 &  & 2.001& &\\\hline
\multirow{5}{*}{3} & 5 & 2778 & 6.604e-04 & (4.163e-04, 9.045e-04) & 6.860e-04 & (4.737e-04,
8.982e-04)&0.2553\\
& 10 & 5996 & 6.177e-04 & (4.141e-04, 8.213e-04) & 6.397e-04 & (4.277e-04,
8.518e-04)&1.434\\
& 15 & 8809 & 5.857e-05 & (3.097e-05, 9.018e-05) & 5.596e-05 & (3.560e-05,
7.633e-05)&3.484\\
& 20 & 12018 & 3.512e-05 & (1.763e-05, 4.661e-05) & 4.164e-05 & (2.227e-05,
6.101e-05)&4.996\\\cline{2-8}
&  \multicolumn{2}{c}{CR}  &3.109 &  & 3.014& &\\\hline
\multirow{5}{*}{4} & 5 & 2778 & 6.026e-05 & (3.726e-05, 8.325e-05) & 6.456e-05 & (5.147e-05,
7.766e-05)&0.3046\\
& 10 & 5996 & 5.645e-06 & (3.708e-06, 7.582e-06) & 5.717e-05 & (4.407e-05,
7.026e-05)&1.6809\\
& 15 & 8809 & 5.163e-06 & (2.738e-06, 7.587e-06) & 5.069e-06 & (3.620e-06,
6.518e-06)&3.741\\
& 20 & 12018 & 3.001e-06 & (1.602e-06, 4.399e-06) & 3.689e-06 & (2.291e-06,
5.088e-06)&6.966\\\cline{2-8}
&  \multicolumn{2}{c}{CR}  &4.227 &  & 3.894& &\\\hline
\end{tabular}
}\end{table}

Table 2 indicates: (i) The larger time points and simulations, the smaller
error of closed-form solutions and numerical solutions no matter which-step
scheme we utilize. (ii) If the number of time points and the number of
simulations are fixed, the errors of closed-form solutions and numerical
solutions become smaller as steps become bigger. (iii) If one's aim for the
error of closed-form solutions and numerical solutions to reach given
accuracy, one cannot only increase time points and simulations but also adopt
multi-step methods, such as the scheme (\ref{2-32}). In other words, this
paper presents a stable high order method to calculate numerical solutions of BSDEs.

\begin{figure}[h]
\begin{center}
\includegraphics[
trim=0.000000in 0.000000in 0.000000in -0.398623in,
height=2.1126in,
width=6.8077in]{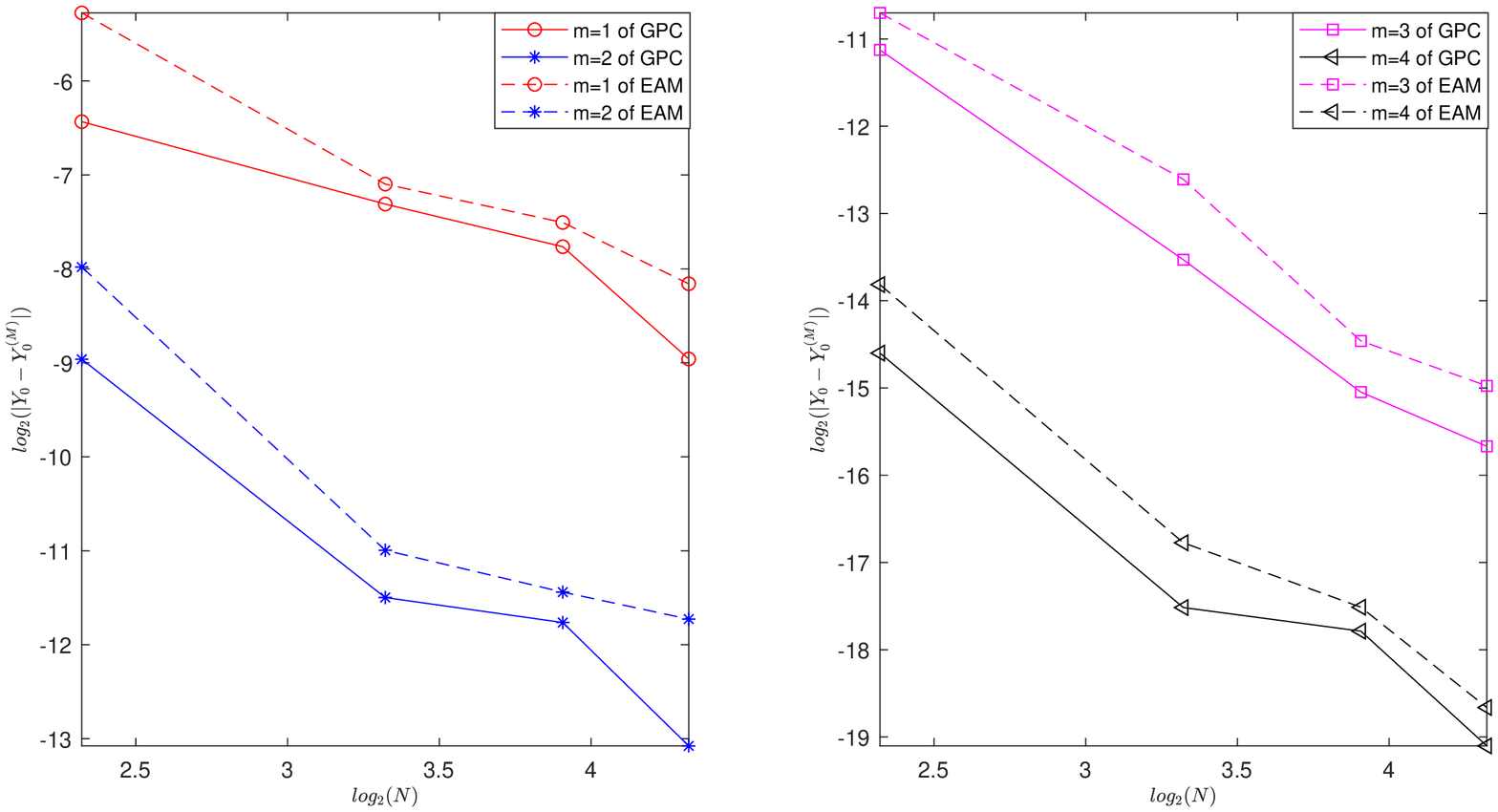}
\end{center}
\caption{The plots of $\log_{2}(|Y_{0}-Y_{0}^{(M)}|)$ versus $\log_{2}(N)$
with GPC scheme and EAM scheme, $M=3000$ }%
\label{Y-CR-RT-algoritnm1}%
\end{figure}\begin{figure}[h]
\begin{center}
\includegraphics[
trim=0.000000in 0.000000in 0.000000in -0.398623in,
height=2.1126in,
width=6.8077in]{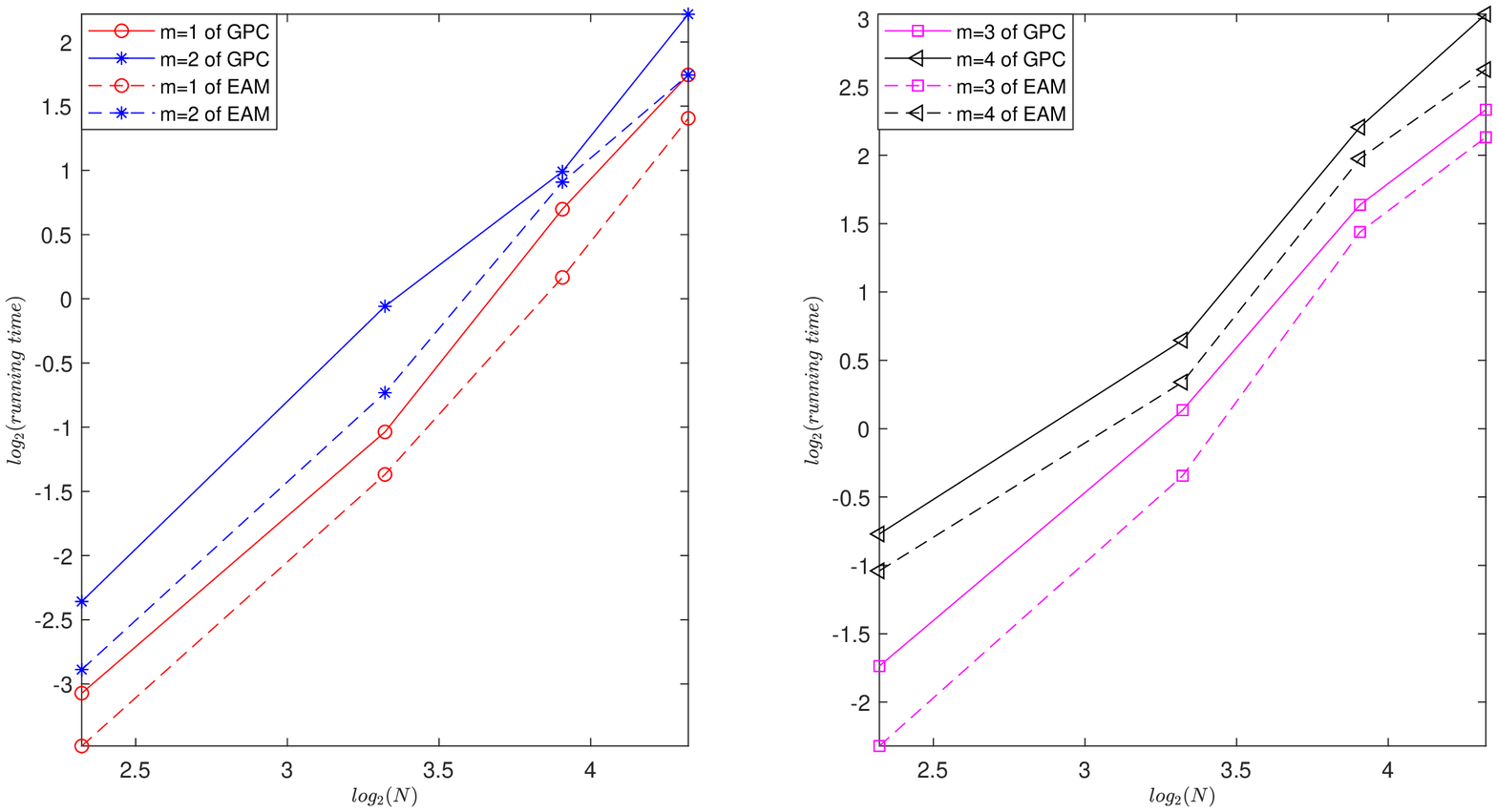}
\end{center}
\caption{The plots of $\log_{2}(running~time)$ versus $\log_{2}(N)$ with GPC
scheme and EAM scheme, $M=3000$ }%
\label{Y-CR-RT-algoritnm1}%
\end{figure}Figure 1 compares the GPC scheme with the EAM scheme in terms of
the accuracy. The left plot in Figure 1 displays the error of $|Y_{0}%
-Y_{0}^{(M)}|$ for the one-step scheme and two-step scheme. The right plot
describes the error of $|Y_{0}-Y_{0}^{(M)}|$ for the three-step scheme and
four-step scheme. Obviously, the accuracy of $Y$ obtained by the GPC scheme is
higher than that of the EAM scheme no matter the number of step is $1,2,3$ or
$4$. Figure 2 compares the GPC scheme with the EAM scheme in terms of the
computational cost. The left plot in Figure 1 displays the running time of
these two methods for the one-step scheme and two-step scheme. The right plot
describes the running time of these two methods for the three-step scheme and
four-step scheme. It is straightforward that the running time of the EAM
scheme is smaller than that of the GPC scheme no matter the number of step is
$1,2,3$ or $4$.

\textbf{Example 2.} Consider the decoupled FBSDEs (taken from \cite{38})
\begin{equation}
\left\{
\begin{array}
[c]{l}%
dX_{t}= \frac{1}{1+2\exp(t+X_{t})}dt + \frac{\exp(t+X_{t})}{1+\exp(t+X_{t}%
)}dW_{t},\\
X_{0} = x,\\
-dY_{t} = \left( -\frac{2Y_{t}}{1+\exp(t+X_{t})} - \frac{1}{2}\left(
\frac{Y_{t}Z_{t}}{1+\exp(t+X_{t})}-Y_{t}^{2}Z_{t}\right) \right) dt
-Z_{t}dW_{t},\\
Y_{T} = \frac{\exp(T+X_{T})}{1+\exp(T+X_{T})},
\end{array}
\right. \label{FBSDEs-detail}%
\end{equation}
with the analytic solutions
\begin{equation}
\left\{
\begin{array}
[c]{l}%
Y_{t} = \frac{\exp(t+X_{t})}{1+\exp(t+X_{t})},\\
Z_{t} = \frac{(\exp(t+X_{t}))^{2}}{(1+\exp(t+X_{t}))^{3}}.
\end{array}
\right. \nonumber
\end{equation}

Take $T=1,x=1, d=2,\widetilde{M}=21,h= \frac{T}{N}$. The basis functions which
are spanned by polynomials whose degree is $2$ are applied to compute the
value of $Y_{i}^{(M)}$ and $Z_{i}^{(M)}$.

\begin{figure}[ptbh]
\begin{center}
\includegraphics[
trim=0.000000in 0.000000in 0.000000in -0.398623in,
height=1.90126in,
width=6.8077in]{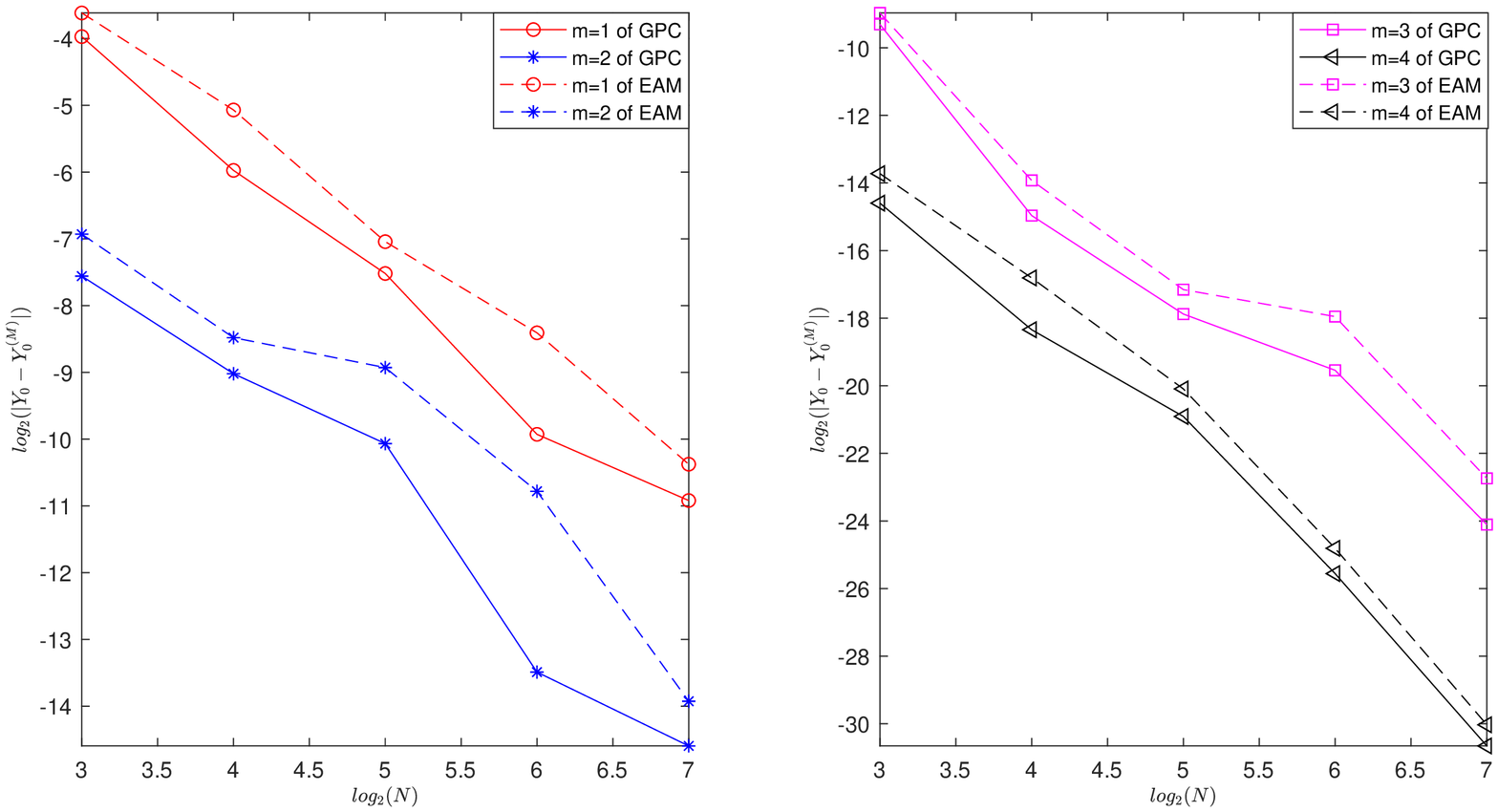}
\end{center}
\par
\caption{The plots of $\log_{2}(|Y_{0}-Y_{0}^{(M)}|)$ versus $\log_{2}(N)$
with GPC scheme and EAM scheme, $M=10000$ }%
\end{figure}\begin{figure}[ptbh]
\begin{center}
\includegraphics[
trim=0.000000in 0.000000in 0.000000in -0.398623in,
height=1.90126in,
width=6.8077in]{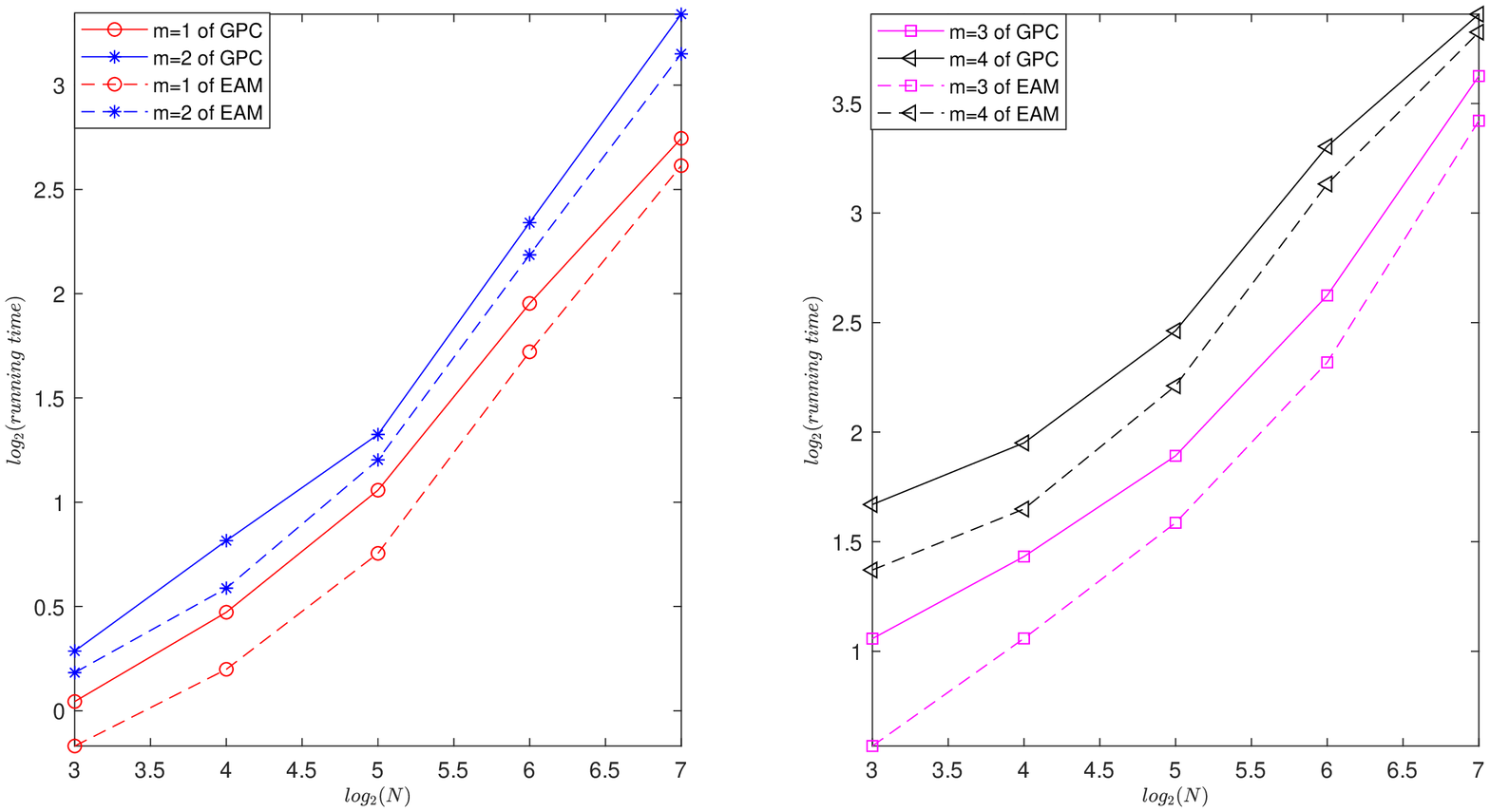}
\end{center}
\par
\caption{The plots of $\log_{2}(running~time)$ versus $\log_{2}(N)$ with GPC
scheme and EAM scheme, $M=10000$ }%
\end{figure}Figure 3 compares the GPC scheme with the EAM scheme in terms of
the error of $|Y_{0}-Y_{0}^{(M)}|$. Figure 4 compares the GPC scheme with the
EAM scheme in terms of the computational cost. These two figures imply that
the GPC scheme possesses higher accuracy than the EAM scheme while the running
time of the GPC scheme is bigger than that of the EAM scheme.

In what follows, we illustrate the case in which the condition (\ref{2-36}) is
satisfied and Dahlquist's root condition does not hold. In other words, we
provide unstable numerical scheme for decoupled FBSDE (\ref{2-1}). For $m=2$,
we introduce a two-step scheme as below
\begin{equation}
\left\{
\begin{array}
[c]{rl}%
\widetilde{Y}_{i}^{\pi}= & \mathbb{E}_{i}\Big[3Y_{i+1}^{\pi}-2Y_{i+2}^{\pi}+
\frac{1}{2}hf_{i+1}^{\pi}-\frac{3}{2}hf_{i+2}^{\pi}\Big],\\
Y_{i}^{\pi}= & \mathbb{E}_{i}\Big[3Y_{i+1}^{\pi}-2Y_{i+2}^{\pi}+h\widetilde{f}%
_{i}^{\pi} -\frac{3}{2}hf_{i+1}^{\pi}-\frac{1}{2}hf_{i+2}^{\pi}\Big],\\
Z_{i}^{\pi}= & \mathbb{E}_{i}\big[2Y_{i+1}^{\pi}\frac{(W_{i+1}-W_{i})^{\top}%
}{h} -\frac{1}{2}Y_{i+2}^{\pi}\frac{(W_{i+2}-W_{i})^{\top}}{h}\big].
\end{array}
\right. \label{5-3}%
\end{equation}
The characteristic polynomial of this two-step scheme is $P(\zeta) = \zeta
^{2}- 3\zeta+2$. Its roots $1,2$ do not fulfil Dahlquist's root condition.
That is to say, this two-step scheme is not stable.

\begin{figure}[h]
\begin{center}
\includegraphics[
trim=0.000000in 0.000000in 0.000000in -0.398623in,
height=2.8126in,
width=4.8077in]{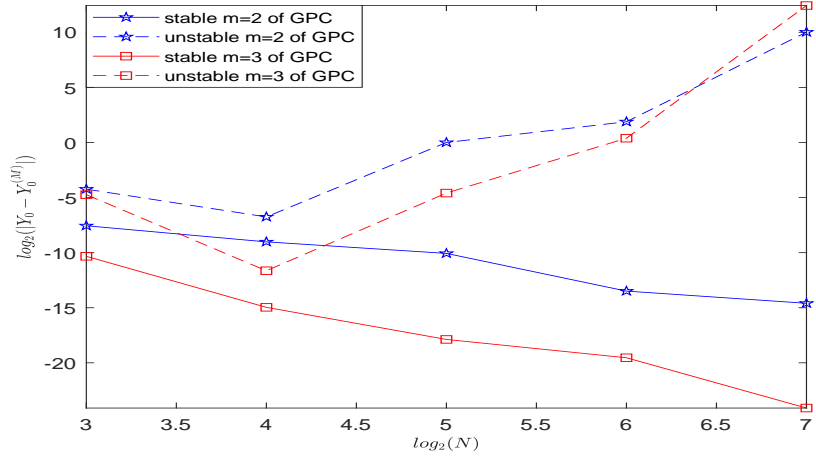}
\end{center}
\caption{The plots of $\log_{2}(|Y_{0}-Y_{0}^{(M)}|)$ versus $\log_{2}(N)$,
$M=10000$ }%
\label{YZ-unstability}%
\end{figure}

For $m=3$, we provide the following three-step scheme
\begin{equation}
\left\{
\begin{array}
[c]{rl}%
\widetilde{Y}_{i}^{\pi}= & \mathbb{E}_{i}\Big[2Y_{i+1}^{\pi}+5Y_{i+2}^{\pi
}-6Y_{i+3}^{\pi} + 2hf_{i+1}^{\pi}-6hf_{i+2}^{\pi}-2hf_{i+3}^{\pi}\Big],\\
Y_{i}^{\pi}= & \mathbb{E}_{i}\Big[2Y_{i+1}^{\pi}+5Y_{i+2}^{\pi}-6Y_{i+3}^{\pi}
-3h\widetilde{f}_{i}^{\pi} +11hf_{i+1}^{\pi}-15hf_{i+2}^{\pi}+hf_{i+3}^{\pi
}\Big],\\
Z_{i}^{\pi}= & \mathbb{E}_{i}\big[3Y_{i+1}^{\pi}\frac{(W_{i+1}-W_{i})^{\top}%
}{h} -\frac{3}{2}Y_{i+2}^{\pi}\frac{(W_{i+2}-W_{i})^{\top}}{h} + \frac{1}%
{3}Y_{i+3}^{\pi}\frac{(W_{i+3}-W_{i})^{\top}}{h}\big].
\end{array}
\right. \label{5-4}%
\end{equation}
The characteristic polynomial of the above scheme is $P(\zeta) = \zeta^{3}-
2\zeta^{2}- 5\zeta+6$. Its roots $-2,1,3$ do not fulfil Dahlquist's root
condition. That is to say, this three-step scheme is not stable.

Figure 5 provides the predictor-corrector method (\ref{2-32}) in terms of the
error of $|Y_{0}-Y_{0}^{(M)}|$. Figure 5 indicates that the variation of
errors is irregular for the unstable two-scheme (the scheme (\ref{5-3})) and
the unstable three-scheme (the scheme (\ref{5-4})). That is to say, both the
scheme (\ref{5-3}) and the scheme (\ref{5-4}) are not stable. Meanwhile,
Figure 5 shows that the errors of $|Y_{0}-Y_{0}^{(M)}|$ become smaller with
the time step sizes $N$ increasing for the stable two-scheme and the stable
three-scheme (These two schemes come from \textbf{Example 1}). In other words,
we verify that the given stable two-scheme and stable three-scheme are indeed
stable by means of a numerical example.

\end{document}